\documentclass{amsart}
\usepackage[normalem]{ulem}
\usepackage{amssymb,epsfig,mathrsfs,mathpazo,url}
\usepackage{bbm}
\usepackage{textcomp}
\usepackage{enumitem}
\usepackage{color}
\usepackage{epsfig,bm}
\usepackage[multiple]{footmisc}
\usepackage{accents} 
\usepackage[T3,T1]{fontenc}
\DeclareSymbolFont{tipa}{T3}{cmr}{m}{n}
\DeclareMathAccent{\invbreve}{\mathalpha}{tipa}{16}

\newtheorem{theorem}{Theorem}[section]

\newtheorem{proposition}[theorem]{Proposition}

\theoremstyle{definition}

\theoremstyle{remark}
\newtheorem{remark}[theorem]{Remark}

\numberwithin{equation}{section}

\newcommand{\eps}{\varepsilon}

\newcommand{\R}{\mathbb R}

\newcommand{\cE}{\mathcal E}

\newcommand{\cL}{\mathcal L}
\newcommand{\cS}{\mathcal S}
\newcommand{\Lis}{\cL\mathrm{is}}

\newcommand{\identity}{\mathrm{Id}}

\DeclareMathOperator{\ran}{ran}

\DeclareMathOperator{\diam}{diam}
\DeclareMathOperator*{\argmin}{argmin}

\DeclareMathOperator{\essinf}{ess\,inf}

\DeclareMathOperator{\divv}{div}

\newcommand{\dif}{\mathop{}\!\mathrm{d}}

\newcommand{\nrm}{| \! | \! |}

\newcommand{\new}[1]{{#1}}

\newcommand{\be}{\begin{equation}}
\newcommand{\ee}{\end{equation}}

\newcommand{\tria}{{\mathcal T}}

\newcommand{\uumlaut}{{\"u}}

{\par\begin{samepage}%
\begin{tabbing}\ttfamily%
 \hspace*{5mm}\=\hspace{3ex}\=\hspace{3ex}\=\hspace{3ex}\=\hspace{3ex}%
\=\hspace{3ex}\=\hspace{3ex}\=\hspace{3ex}\=\hspace{3ex}\kill}%
{\end{tabbing}\end{samepage}}

\makeatletter
\@namedef{subjclassname@2020}{%
  \textup{2020} Mathematics Subject Classification}
\makeatother

\title[Minimal residual space-time discretizations of parabolic equations]{Minimal residual space-time discretizations of parabolic equations: Asymmetric spatial operators}

\date{\today}

\author{Rob Stevenson, Jan Westerdiep}

\address{
Korteweg--de Vries (KdV) Institute for Mathematics, University of Amsterdam, P.O. Box 94248, 1090 GE Amsterdam, The Netherlands.
}
\email{r.p.stevenson@uva.nl, j.h.westerdiep@uva.nl}

\thanks{The second author has been supported by the Netherlands Organization for Scientific Research (NWO) under contract.~no.~613.001.652}


\begin{document}

\begin{abstract} We consider a minimal residual discretization of a simultaneous space-time variational formulation of parabolic evolution equations. Under the usual `LBB' stability condition on pairs of trial- and test spaces we show quasi-optimality of the numerical approximations without assuming symmetry of the spatial part of the differential operator. Under a stronger LBB condition we show error estimates in an energy-norm that are independent of this spatial differential operator.
\end{abstract}

\maketitle
\section{Introduction}
This paper is about the numerical solution of parabolic evolution equations in a simultaneous space-time variational formulation.
Compared to classical time-stepping schemes, simultaneous space-time methods are much better suited for a massively parallel implementation (e.g.~\cite{234.7,306.65}), allow for local refinements in space and time (e.g.~\cite{249.4,75.38,249.991,306.6}), and produce numerical approximations from the employed trial spaces that are quasi-best.

The standard bilinear form that results from a space-time variational formulation is non-coercive, which makes it difficult to construct pairs of discrete trial and test spaces that inherit the stability of the continuous formulation. For this reason, in~\cite{11} R.~Andreev proposed to use minimal residual discretizations. They have an equivalent interpretation as Galerkin discretizations of an extended self-adjoint, but indefinite, mixed system having as secondary variable
 the Riesz lift of the residual of the primal variable w.r.t.~the PDE\@.

For pairs of trial spaces that satisfy a Ladyzhenskaja--Babu\u{s}ka--Brezzi (LBB) condition, it was shown that w.r.t.~the norm on the natural solution space, being an intersection of Bochner spaces, the Galerkin solutions are quasi-best approximations from the selected trial spaces. This LBB condition was verified in~\cite{11} for `full' and `sparse' tensor products of various finite element spaces in space and time. The sparse tensor product setting was then generalized in~\cite[Proposition 5.1]{249.991} to allow for local refinements in space and time whilst retaining (uniform) LBB stability.

A different minimal residual formulation of first order system type was introduced in~\cite{75.257}, see also~\cite{75.28}. Here the various residuals are all measured in $L_2$-norms, meaning that they do not have to be introduced as separate variables, and the resulting bilinear form is coercive.

Closer in spirit to~\cite{11} are the space-time methods presented in~\cite{249.2,169.05,35.825}, in which error bounds are presented w.r.t.~mesh-dependent norms.
In~\cite{64.5765,249.3} space-time variational methods are presented that lead to coercive bilinear forms based on fractional Sobolev norms of order $\frac12$. A first order space-time DPG formulation of the heat equation is presented in~\cite{64.585}.

A restriction imposed in~\cite{11}, as well as in the other mentioned references apart from~\cite{35.825, 75.28}, is that the spatial part of the PDO is not only coercive but also symmetric. In~\cite{249.99} we could remove the symmetry condition for the analysis of a related Br\'{e}zis--Ekeland--Nayroles (BEN) (\cite{35.81,233.5}) formulation of the parabolic PDE\@.
In the current work, we prove that also for the minimal residual (MR) method the symmetry condition can be dropped.
So for both MR and BEN we show that under the aforementioned LBB condition the Galerkin approximations are quasi-optimal, where the bound on the error in the numerical approximation for BEN improves upon the one from~\cite{249.99}.

The error bounds for both MR and BEN degrade for increasing asymmetry. This is not an artefact of the theory but is confirmed by numerical experiments.
Under a stronger LBB condition on the pair of trial spaces, however, we will prove that the MR and BEN approximations are quasi-best w.r.t.~a continuous, i.e.,  `mesh-\emph{in}dependent', energy-norm, uniformly in the spatial PDO\@.

We present numerical tests for the evolution problem governed by the simple PDE $\partial_t-\eps\partial_x^2+\partial_x+e\identity$ on $(0,1)^2$ with initial and boundary conditions, where $e$ is either $0$ or $1$.
For the case that homogeneous Dirichlet boundary conditions are prescribed at the outflow boundary $x=1$, the results for very small $\eps$
illustrate that quasi-optimal approximations do no necessarily mean accurate approximations. Indeed the error in the computed solution is large because of the unresolved boundary layer. The minimization of the error in the energy-norm of least squares type causes a global spread of the error along the streamlines. We tackled this problem by imposing these boundary conditions only weakly.

\subsection{Organization}
In Sect.~\ref{S2} we recall the well-posed space-time variational formulation of the parabolic problem and study its conditioning.
Under the usual LBB condition, in Sect.~\ref{S3} we show quasi-optimality of the MR method without assuming symmetry of the spatial differential operator.
A similar result is shown for BEN in Sect.~\ref{S4}.
Known results concerning the verification of this LBB condition are summarized in Sect.~\ref{S5}, together with results about optimal preconditioning.

In Sect.~\ref{S6} we equip the solution space with an energy-norm, and, under a stronger LBB condition, show error estimates for MR and BEN that are independent of the spatial differential operator.
We present an a posteriori error estimator that, under an even stronger  LBB condition, is efficient and, modulo a date-oscillation term, is reliable.

In Sect.~\ref{Sconcrete} we apply the general theory to the example of the  convection-diffusion problem.
We give pairs of trial- and test spaces that satisfy the 2nd and 3rd mentioned LBB conditions. Finally, in Sect.~\ref{Snumer} we present numerical results for the MR method in the simple case of having a one-dimensional spatial domain. To solve the problems caused by an unresolved boundary layer, we modify the method by imposing a boundary condition weakly.

\subsection{Notations}\label{ssec:1.2}
In this work, by $C \lesssim D$ we will mean that $C$ can be bounded by a multiple of $D$, independently of parameters that $C$ and $D$ may depend on.
Obviously, $C \gtrsim D$ is defined as $D \lesssim C$, and $C\eqsim D$ as $C\lesssim D$ and $C \gtrsim D$.

For normed linear spaces $E$ and $F$, by $\cL(E,F)$ we will denote the normed linear space of bounded linear mappings $E \rightarrow F$,
and by $\Lis(E,F)$ its subset of boundedly invertible linear mappings $E \rightarrow F$.
We write $E \hookrightarrow F$ to denote that $E$ is continuously embedded into $F$.
For simplicity only, we exclusively consider linear spaces over the scalar field $\R$.

\section{Well-posed variational formulation}\label{S2}
Let $V,H$ be separable Hilbert spaces of functions on some ``spatial domain''
such that $V \hookrightarrow H$ with dense embedding.
Identifying $H$ with its dual, we obtain the Gelfand triple
$V \hookrightarrow H \simeq H' \hookrightarrow V'$. We use $\langle \cdot,\cdot\rangle$ to denote both the scalar product on $H \times H$ as well as its unique extension to the duality pairing on $V' \times V$ or $V \times V'$, and denote the norm on $H$ by $\|\cdot\|$.

For a.e.
$$
t \in I:=(0,T),
$$
let $a(t;\cdot,\cdot)$ denote a bilinear form on $V \times V$ such that for
any $\eta,\zeta \in V$, $t \mapsto a(t;\eta,\zeta)$ is measurable on $I$,
and such that for some $\varrho \in \R$, for
a.e. $t\in I$,
\begin{alignat}{3} \label{boundedness}
|a(t;\eta,\zeta)|
& \lesssim  \|\eta\|_{V} \|\zeta\|_{V} \quad &&(\eta,\zeta \in V) \quad &&\text{({\em boundedness})},
\\ \label{garding}
 a(t;\eta,\eta) +\varrho\langle \eta,\eta\rangle &\gtrsim \|\eta\|_{V}^2 \quad
&&(\eta \in {V}) \quad &&\text{({\em G{\aa}rding inequality})}.
\end{alignat}

With $A(t) \in \Lis({V},V')$ being defined by $ (A(t) \eta)(\zeta):=a(t;\eta,\zeta)$, given a forcing function $g$ and an initial value $u_0$, we are interested in solving the {\em parabolic initial value problem} to finding $u$ such that
\begin{equation} \label{11}
\left\{
\begin{array}{rl}
\frac{\dif u}{\dif t}(t) +A(t) u(t)&\!\!\!= g(t) \quad(t \in I),\\
u(0) &\!\!\!= u_0.
\end{array}
\right.
\end{equation}

In a simultaneous space-time variational formulation, the parabolic PDE reads as finding $u$ from a suitable space of functions $X$ of time and space such that
$$
(Bw)(v):=\int_I
\langle { \textstyle \frac{\dif w}{\dif t}}(t), v(t)\rangle +
a(t;w(t),v(t)) \dif t = \int_I
\langle g(t), v(t)\rangle\dif t =:g(v)
$$
for all $v$ from another suitable space of functions $Y$ of time and space.
One  possibility to enforce the initial condition is by testing it against additional test functions. A proof of the following result can be found in \cite{247.15}, \new{cf.
\cite[Ch.~3, Thm.~4.1]{185},
\cite[Ch.~IV, \S26]{314.9},
\cite[Ch.XVIII, \S3]{63}, and
\cite[Thm.~6.6]{eg04}} for similar statements.

\begin{theorem} \label{thm0} With $X:=L_2(I;{V}) \cap H^1(I;V')$, $Y:=L_2(I;{V})$,
under conditions \eqref{boundedness} and \eqref{garding} it holds that
$$
\new{(B,\gamma_0)} \in \Lis(X,Y' \times H),
$$
where for $t \in \bar{I}$, $\gamma_t\colon u \mapsto u(t,\cdot)$ denotes the trace map.
That is, assuming $g \in Y'$ and $u_0 \in H$, finding $u \in X$ such that
\be \label{x12}
\new{(B,\gamma_0) =(g,u_0)}
\ee
is a well-posed simultaneous space-time variational formulation of \eqref{11}.
\end{theorem}

With $\tilde{u}(t):=u(t) e^{-\varrho t}$, \eqref{11} is equivalent to $
\frac{\dif \tilde{u}}{\dif t}(t) +(A(t)+\varrho \identity) \tilde{u}(t)= g(t)e^{-\varrho t}$  ($t \in I$),
$\tilde{u}(0) = u_0$. Since $((A(t)+\varrho \identity)\eta)(\eta) \gtrsim \|\eta\|_{V}^2$, w.l.o.g.~we will always assume that,
besides \eqref{boundedness},
\eqref{garding} is valid for $\varrho=0$, i.e., for a.e.~$t \in I$,
\be \label{coercivity}
 a(t;\eta,\eta) \gtrsim \|\eta\|_{V}^2 \quad
(\eta \in {V}) \quad \text{({\em coercivity})}.
\ee

We define $A, A_s \in \Lis(Y,Y')$, $A_a \in \cL(Y,Y')$, and $C, \partial_t \in \cL(X,Y')$ by
\begin{align*}
&(Aw)(v):=\int_I a(t;w(t),v(t))\dif t,\quad
A_s:={\textstyle  \frac12}(A+A'), \quad A_a:={\textstyle \frac12}(A-A'),\\
&\partial_t:=B-A,\quad C:=B-A_s=\partial_t + A_a,
\end{align*}
and equip $Y$ with `energy'-scalar product $\langle\cdot,\cdot\rangle_Y:=(A_s \cdot)(\cdot)$, and norm
$$
\|v\|_Y:=\sqrt{(A_s v)(v)}.
$$
being, thanks to \eqref{boundedness} and \eqref{coercivity}, equivalent to the standard norm on $Y$.
Equipping $Y'$ with the resulting dual norm, $A_s \in \Lis(Y,Y')$ is an isometric isomorphism, and so for $f \in Y'$ we have
$$
f(A_s^{-1}f)=(A_s A_s^{-1}f)(A_s^{-1}f)=\|A_s^{-1}f\|^2_{Y}=\|f\|_{Y'}^2.
$$

For some constant $\beta \geq 1$, we equip $X$ with norm
$$
\|\cdot\|_X:=\sqrt{\|\cdot\|_Y^2+\|\partial_t \cdot\|_{Y'}^2+\|\gamma_T \cdot\|^2+(\beta-1) \|\gamma_0 \cdot\|^2},
$$
being, thanks to $X \hookrightarrow C(\overline{I};H)$, equivalent to the standard norm on $X$.
In addition, we define the energy-norm on $X$ by
$$
\nrm \cdot\nrm_X:=\sqrt{\|B \cdot\|_{Y'}^2+\beta\|\gamma_0 \cdot\|^2},
$$
which, thanks to Theorem~\ref{thm0}, is indeed a norm on $X$.

\begin{proposition} \label{prop:continuous} With $\alpha:=\|A_a\|_{\cL(Y,Y')}$, for $0 \neq w \in X$ it holds that
$$
\Big(1+\frac{\alpha}{2}\big(\alpha+\sqrt{\alpha^2+4}\,\big)\Big)^{-1} \leq \frac{\nrm w\nrm_X^2}{\|w\|_X^2}
\leq 1+\frac{\alpha}{2}\big(\alpha+\sqrt{\alpha^2+4}\,\big),
$$
so that, in particular, both norms are equal when $A_a=0$.
\end{proposition}

\begin{proof}
Using that for  $w,v \in X$,
\begin{align*}
((\partial_t+\partial_t'+\gamma_0'\gamma_0)w)(v)&=
\int_I \langle { \textstyle \frac{\dif w}{\dif t}}(t), v(t)\rangle+\langle w(t), { \textstyle \frac{\dif v}{\dif t}}(t)\rangle\dif t+\langle w(0), v(0)\rangle\\
&=
\int_I {\textstyle \frac{\dif}{\dif t}} \langle w(t),v(t)\rangle\dif t+\langle w(0),v(0)\rangle=(\gamma_T'\gamma_Tw)(v),
\end{align*}
we find that
\be \label{BC}
\begin{split}
B' A_s^{-1} B + \beta \gamma_0' \gamma_0&=(C'+A_s)A_s^{-1}(C+A_s)+ \beta \gamma_0' \gamma_0\\
&=C' A_s^{-1} C+A_s+C'+C+\beta\gamma_0' \gamma_0\\
&=C' A_s^{-1} C+A_s+\partial_t'+\partial_t+\beta \gamma_0' \gamma_0\\
&=C' A_s^{-1} C + A_s +\gamma_T' \gamma_T+(\beta-1) \gamma_0' \gamma_0.
\end{split}
\ee
For $w \in X$,
$$
(C' A_s^{-1} C w)(w)=(Cw)(A_s^{-1} Cw)=\|(\partial_t+A_a)w\|_{Y'}^2\leq (\|\partial_t w\|_{Y'}+\alpha \|w\|_{Y})^2,
$$
and so, for any $\eta \neq 0$, Young's inequality shows that
\begin{align*}
\|Bw\|_{Y'}^2&+\beta\|\gamma_0 w\|^2=\big((C' A_s^{-1} C + A_s +\gamma_T' \gamma_T+(\beta-1) \gamma_0' \gamma_0)(w)\big)(w) \\
& \leq (1+\eta^2)\|\partial_t w\|_{Y'}^2+((1+\eta^{-2})\alpha^2+1)\|w\|_Y^2+\|\gamma_T w\|^2+(\beta-1)\|\gamma_0 w\|^2.
\end{align*}
Solving $(1+\eta^2)=(1+\eta^{-2})\alpha^2+1$ gives $1+\eta^2=1+{\textstyle \frac{\alpha}{2}}\big(\alpha+\sqrt{\alpha^2+4}\,\big)$, showing one of the bounds of the statement.

From
$$
\|(\partial_t+A_a)w\|_{Y'}^2\geq (\|\partial_t w\|_{Y'}-\alpha \|w\|_{Y})^2\geq (1-\eta^2)\|\partial_t w\|^2_{Y'}+(1-\eta^{-2})\alpha^2 \|w\|_{Y}^2
$$
again by Young's inequality, by solving $\eta^2$ from $1-\eta^2=(1-\eta^{-2})\alpha^2+1$ the other bound follows.
\end{proof}

\begin{remark} \label{relative}
Because $\|\cdot\|_Y$ is defined in terms of the symmetric part $A_s$ of the spatial differential operator $A$,
 $\alpha=\|A_a\|_{\cL(Y,Y')}$ is a measure for the \emph{relative} asymmetry of the operator $A$.
 Indeed
$\|A_a\|_{\cL(Y,Y')}=\|A_s^{-\frac12} A_aA_s^{-\frac12} \|_{\cL(L_2(I;H),L_2(I;H))}=\rho(A_s^{-\frac12} A'_a A_s^{-1} A_a A_s^{-\frac12})^{\frac12}=\rho(A_s^{-1} A_a A_s^{-1} A_a)^{\frac12}$, where we used that $A_a' = -A_a$.
 \end{remark}

 \new{A result on the conditioning of $(B,\gamma_0) \in \Lis(X,Y'\times H)$ similar to Proposition~\ref{prop:continuous} but w.r.t.~different norms on $X$ and $Y$ can be found in \cite[Lemmas~71.1 \& 71.2]{70.99}.}

\section{Minimal residual (MR) method} \label{S3}
Let $(X^\delta,Y^\delta)_{\delta \in \Delta}$ a family of closed, non-zero subspaces of $X$ and $Y$, respectively. For $\delta \in \Delta$, let $E_X^\delta$ and $E_Y^\delta$ denote the trivial embeddings $X^\delta \rightarrow X$ and $Y^\delta \rightarrow Y$, which we \new{sometimes} write for clarity, \new{but that we mainly introduce because of their duals}. We assume that
\begin{align} \label{cond1}
&X^\delta \subseteq Y^\delta\quad(\delta \in \Delta),\\
\label{cond2}
&\gamma_\Delta^{\partial_t}:=\inf_{\delta \in \Delta} \inf_{\{w \in X^\delta\colon \partial_t E_X^\delta w \neq 0\}} \frac{\|{E_Y^\delta}' \partial_t E_X^\delta w\|_{{Y^\delta}'}}{\|\partial_t E_X^\delta w\|_{Y'}} >0.
\end{align}

Furthermore, for efficiency reasons we assume to have available a $K_Y^\delta={K_Y^\delta}' \in \Lis({Y^\delta}',Y^\delta)$ (a `preconditioner'), such that for some constants $0<r_\Delta\leq R_\Delta<\infty$,
\be \label{cond3}
\frac{((K_Y^\delta)^{-1} v)(v)}{({E^\delta_Y}'A_s E^\delta_Y v)(v)} \in [r_\Delta,R_\Delta] \quad(\delta \in \Delta,\,v \in Y^\delta),
\ee
or, equivalently,
$\frac{f(K_Y^\delta f)}{f(({E^\delta_Y}'A_s E^\delta_Y)^{-1}f)} \in [R_\Delta^{-1},r_\Delta^{-1}]$ ($\delta \in \Delta,f \in {Y^\delta}'$).

Noticing that $\|f\|_{{Y^\delta}'}^2=f(({E^\delta_Y}'A_s E^\delta_Y)^{-1}f)$, the expression
$$
\|\cdot\|_{K_Y^\delta}:=\sqrt{(\cdot)(K_Y^\delta \cdot)}
$$
defines an equivalent norm on ${Y^\delta}'$, and our Minimal Residual approximation $u^\delta \in X^\delta$ of the solution $u \in X$ of \eqref{x12} is defined as
\be \label{ls}
u^\delta:=\argmin_{w \in X^\delta} \|{E_Y^\delta}'(B E_X^\delta w-g)\|_{K_Y^\delta}^2+\beta \|\gamma_0 E_X^\delta w-u_0\|^2,
\ee
for some constant $\beta \geq 1$. Later we will see that, thanks to \eqref{cond2} and \eqref{cond3},
\be \label{e1}
\inf_{0\neq w \in X^\delta} \sup_{(v_1,v_2) \in Y^\delta \times H} \frac{(B E_X^\delta w)(E_Y^\delta  v_1)+\beta \langle \gamma_0 E_X^\delta w,v_2\rangle}{\sqrt{((K_Y^\delta)^{-1} v_1)(v_1)+\beta\|v_2\|^2}}>0
\ee
(even uniformly in $\delta \in \Delta$)\footnote{This follows by combining \eqref{e7}, \eqref{e8}, and \eqref{e9}.} which implies that \eqref{ls} has a unique solution.
The numerical approximation \eqref{ls} was proposed in \cite{11}\footnote{In \cite{11}, the norm $\|\gamma_0 E_X^\delta w-u_0\|$ reads as $\sup_{0 \neq z \in Z^\delta}\frac{\langle\gamma_0 E_X^\delta w-u_0,z\rangle}{\|z\|}$ for some $H \supseteq Z^\delta \supseteq \ran\gamma_0|_{X^\delta}$ which generalization seems not very helpful.}, and further investigated in \cite{249.99}.
In both these references the analysis of the MR method was restricted to the case that $A_a=0$.
The introduction of the parameter $\beta \geq 1$ allows to appropriately weight both terms in the least squares minimization.

The solution $u^\delta$ of the MR problem is the solution of the resulting Euler--Lagrange equations, which read as
\be \label{MRM}
({E_X^\delta}' B' E_Y^\delta K_Y^\delta {E_Y^\delta}' B E_X^\delta+{E_X^\delta}' \beta \gamma_0' \gamma_0 E_X^\delta) u^\delta=
{E_X^\delta}' B' E_Y^\delta K_Y^\delta {E_Y^\delta}' g+{E_X^\delta}'\beta  \gamma_0' u_0,
\ee
as also the second component of the solution $(\mu^\delta,u^\delta) \in Y^\delta \times X^\delta$ of
\be \label{m9}
\left[\begin{array}{@{}ccc@{}}(K_Y^\delta)^{-1} & {E_Y^{\delta}}' B E^\delta_X\\ {E^\delta_X}' B' E_Y^{\delta}& -{E^\delta_X}' \beta \gamma_0' \gamma_0 E^\delta_X \end{array}\right]
\left[\begin{array}{@{}c@{}} \mu^{\delta} \\ u^{\delta} \end{array}\right]=
\left[\begin{array}{@{}c@{}} {E^{\delta}_Y}' g \\ -{E^\delta_X}'\beta  \gamma_0' u_0 \end{array}\right],
\ee
being a useful representation when no efficient preconditioner is available and one has to resort to $(K_Y^\delta)^{-1}={E^\delta_Y}'A_s E^\delta_Y$.

With the ``projected'' or ``approximate''  (because generally $Y^\delta \neq Y$) \emph{trial-to-test operator} $T^\delta \new{=(T^\delta_1,T^\delta_2)} \in \cL(X,Y^\delta \times H)$ defined by
\be \label{e3}
((K_Y^\delta)^{-1} T_{\new{1}}^\delta w)(v_1)+\beta\langle T_{\new{2}}^\delta w,v_2\rangle=(B w)(E_Y^\delta v_1)+\beta \langle \gamma_0 u,v_2\rangle \quad((v_1,v_2) \in Y^\delta \times H),
\ee
and \new{the} ``projected'' or ``approximate'' \emph{optimal test space} $Z^\delta:=\ran T^\delta|_{X^\delta}$, a third equivalent formulation of \eqref{ls}  (see e.g.~\cite{64.14}, \cite[Prop.~2.2]{35.8565}, \cite{64.17})  is finding $u^\delta \in X^\delta$ that solves the Petrov--Galerkin system
\be \label{e2}
(B E_X^\delta u^\delta)(E_Y^\delta v_1)+\beta \langle \gamma_0 E_X^\delta  u^\delta,v_2\rangle =g(E_Y^\delta v_1)+\beta \langle u_0,v_2\rangle \quad((v_1,v_2) \in Z^\delta).
\ee
Note that~\eqref{e2} avoids the `normal equations'~\eqref{MRM}. It
will allow us to derive a quantitatively sharp estimate for the error in $u^\delta$.
From \eqref{cond3} and  \eqref{e1}, one infers that $\sup_{0 \neq w \in X^\delta} \frac{\|T^\delta w\|_{Y \times H}}{\|w\|_X}>0$, so that, thanks to $X^\delta$ being closed, $Z^\delta$ is a closed subspace of $Y^\delta \times H$.
We orthogonally decompose $Y^\delta \times H$ into $Z^\delta$ and $(Z^\delta)^\perp$, where here we equip $Y^\delta$ with inner product $((K_Y^\delta)^{-1} \cdot)(\cdot)$.
From \eqref{e3} one infers that for $w \in X^\delta$ and $(v_1,v_2) \in (Z^\delta)^\perp$, it holds that
$(B w)(v_1)+\beta \langle \gamma_0 u,v_2\rangle=0$, and so
\be \label{e4}
\sup_{(v_1,v_2) \in Y^\delta \times H} \hspace*{-.7em} \frac{(B E_X^\delta w)(E_Y^\delta v_1)+\beta \langle \gamma_0 E_X^\delta w,v_2\rangle}{\sqrt{((K_Y^\delta)^{-1} v_1)(v_1)+\beta\|v_2\|^2}}=
 \hspace*{-.7em} \sup_{(v_1,v_2) \in Z^\delta}  \hspace*{-.7em}\frac{(B E_X^\delta w)(E_Y^\delta v_1)+\beta \langle \gamma_0E_X^\delta w,v_2\rangle}{\sqrt{((K_Y^\delta)^{-1} v_1)(v_1)+\beta\|v_2\|^2}}.
\ee


\begin{theorem} \label{thm:mainMRM} Under conditions \eqref{cond1}, \eqref{cond2}, and \eqref{cond3}, the solution $u^\delta \in X^\delta$ of \eqref{MRM} exists uniquely, and satisfies
$$
\|u-u^\delta\|_X \leq
\sqrt{{\textstyle \frac{\max(R_\Delta,1) \Big(1+{\textstyle \frac12}\big(\alpha^2+\alpha\sqrt{\alpha^2+4}\,\big)\Big)}
{\min(r_\Delta,1) {\textstyle \frac12}\Big((\gamma_\Delta^{\partial_t})^2+\alpha^2+1-\sqrt{((\gamma_\Delta^{\partial_t})^2+\alpha^2+1)^2-4 (\gamma_\Delta^{\partial_t})^2}\Big)}}}
 \inf_{w \in X^\delta} \|u-w\|_X.
$$
\end{theorem}

\new{Before we give its proof, we make a few comments on this error bound.
First, it shows that for $\gamma_\Delta^{\partial_t}=r_\Delta=R_\Delta=1$ and $\alpha=0$, $u^\delta$ is the best approximation to $u$ from $X^\delta$.
Secondly, for $\alpha=0$ (and $\beta=1$), the bound equals the one found in \cite[Thm.~3.7 \& Rem.~3.8]{249.99}.
Thirdly, using Mathematica$^{\text{\textregistered}}$ \cite{310.7} we find  that\footnote{\new{\texttt{Reduce[\{Sqrt[3]/2 <= Sqrt[(1 + 1/2*(a\textasciicircum2 + a*Sqrt[a\textasciicircum2 + 4]))/(1/2*(g\textasciicircum2 + a\textasciicircum2 + 1 - Sqrt[(g\textasciicircum2 + a\textasciicircum2 + 1)\textasciicircum2 - 4g\textasciicircum2]))] / ((1 + 1/2*(a\textasciicircum2 + a*Sqrt[a\textasciicircum2 + 4]))/g) <= 1\}, \{a, g\}]} returns \texttt{a >= 0 \&\& 0 < g <= 1}.}}
$$
{\textstyle \sqrt{ \frac{ \Big(1+{\textstyle \frac12}\big(\alpha^2+\alpha\sqrt{\alpha^2+4}\,\big)\Big)}
{ {\textstyle \frac12}\Big((\gamma_\Delta^{\partial_t})^2+\alpha^2+1-\sqrt{((\gamma_\Delta^{\partial_t})^2+\alpha^2+1)^2-4 (\gamma_\Delta^{\partial_t})^2}\Big)}}}
\, \Big/\,
 {\textstyle \frac{1+{\textstyle \frac12}\big(\alpha^2+\alpha\sqrt{\alpha^2+4}\,\big)}{\gamma_\Delta^{\partial_t}}} \in \big[\tfrac12 \sqrt{3},1\big]
$$
for $\alpha \ge 0$, $\gamma_\Delta^{\partial_t}\in (0,1]$,
clarifying the behaviour of the bound  in terms of $\alpha$ and $\gamma_\Delta^{\partial_t}$.}

\begin{proof} Let $u$ be the solution of \eqref{x12}, i.e., $g=Bu$ and $u_0=\gamma_0 u$.
The mapping $P^\delta \in \cL(X,X)$ from $u$ to the solution
  $u^\delta \in X^\delta$ of \eqref{ls} or, equivalently, \eqref{MRM} or \eqref{e2}, is a projector onto $X^\delta$ that, by our assumption $X^\delta \not\in \{0,X\}$, is unequal to $0$ or $\identity$.
Consequently $\|P^\delta\|_{\cL(X,X)}=\|\identity -P^\delta\|_{\cL(X,X)}$ (\cite{169.5,315.7}), and
\be \label{e5}
\begin{split}
\|u-u^\delta\|_X & =\|(\identity-P^\delta)u\|_X=\inf_{w \in X^\delta}\|(\identity-P^\delta)(u-w)\|_X\\
& \leq \|P^\delta\|_{\cL(X,X)}\inf_{w \in X^\delta}\|u-w\|_X.
\end{split}
\ee

  To bound $\|P^\delta\|_{\cL(X,X)}= \sup_{0 \not= w \in X} \tfrac{\|P^\delta w\|_X}{\|w\|_X}$, given $w \in X$, let $E^\delta_X  w^\delta:=P^\delta w$.
Using \eqref{cond3}, \eqref{e4}, \eqref{e2}, and Proposition~\ref{prop:continuous} we estimate
\be \label{e6}
\begin{split}
& \sup_{(v_1,v_2) \in Y^\delta \times H} \frac{\big((B E_X^\delta w^\delta)(E_Y^\delta v_1)+\beta \langle \gamma_0 E_X^\delta w^\delta,v_2\rangle\big)^2}
{\|E_Y^\delta v_1\|_Y^2+\beta \|v_2\|^2}
\\ &\leq
\tfrac{1}{\min(r_\Delta,1)}
\sup_{(v_1,v_2) \in Y^\delta \times H} \frac{\big((B E_X^\delta w^\delta)(E_Y^\delta v_1)+\beta \langle \gamma_0 E_X^\delta w^\delta,v_2\rangle\big)^2}
{((K_Y^\delta)^{-1}v_1)(v_1)+\beta \|v_2\|^2}
\\ &= \tfrac{1}{\min(r_\Delta,1)}
\sup_{(v_1,v_2) \in Z^\delta} \frac{\big((B E_X^\delta w^\delta)(E_Y^\delta v_1)+\beta \langle \gamma_0 E_X^\delta w^\delta,v_2\rangle\big)^2}
{((K_Y^\delta)^{-1}v_1)(v_1)+\beta \|v_2\|^2}
\\ &= \tfrac{1}{\min(r_\Delta,1)}
\sup_{(v_1,v_2) \in Z^\delta} \frac{\big((B w)(E_Y^\delta v_1)+\beta \langle \gamma_0 w,v_2\rangle\big)^2}
{((K_Y^\delta)^{-1}v_1)(v_1)+\beta \|v_2\|^2}
\\ &\leq
 \tfrac{\max(R_\Delta,1)}{\min(r_\Delta,1)}
\sup_{(v_1,v_2) \in Y \times H} \frac{\big((B w)(v_1)+\beta \langle \gamma_0 w,v_2\rangle\big)^2}
{\|v_1\|_Y^2+\beta \|v_2\|^2} \\
&= \tfrac{\max(R_\Delta,1)}{\min(r_\Delta,1)} \nrm w\nrm_X^2 \leq  \tfrac{\max(R_\Delta,1)}{\min(r_\Delta,1)} \Big(1+{\textstyle \frac12}\big(\alpha^2+\alpha\sqrt{\alpha^2+4}\,\big)\Big) \|w\|_X^2.
\end{split}
\ee

On the other hand,
\be \label{e7}
\begin{split}
& \sup_{(v_1,v_2) \in Y^\delta \times H} \frac{\big((B E_X^\delta w^\delta)(E_Y^\delta v_1)+\beta \langle \gamma_0 E_X^\delta w^\delta,v_2\rangle\big)^2}
  {\|E_Y^\delta v_1\|_Y^2+\beta \|v_2\|^2}\\
&= \sup_{(v_1,v_2) \in Y^\delta \times H}
\frac{\big((A_s E_Y^\delta ({E_Y^\delta}' A_sE_Y^\delta)^{-1} {E_Y^\delta}' B E_X^\delta w^\delta)(E_Y^\delta v_1)+\beta \langle \gamma_0 E_X^\delta w^\delta,v_2\rangle \big)^2}
  {\|E_Y^\delta v_1\|_Y^2+\beta \|v_2\|^2}\\
& =\sup_{(v_1,v_2) \in Y^\delta \times H}
 \frac{\big(\langle E_Y^\delta ({E_Y^\delta}' A_sE_Y^\delta)^{-1} {E_Y^\delta}' B E_X^\delta w^\delta,E_Y^\delta v_1\rangle_Y+\beta \langle \gamma_0 E_X^\delta w^\delta,v_2\rangle \big)^2}
{\|E_Y^\delta v_1\|_Y^2+\beta \|v_2\|^2}\\
&= \|E_Y^\delta ({E_Y^\delta}' A_sE_Y^\delta)^{-1} {E_Y^\delta}' B E_X^\delta w^\delta\|_Y^2+\beta\|\gamma_0 E_X^\delta w^\delta\|^2\\
& = (A_s E_Y^\delta ({E_Y^\delta}' A_sE_Y^\delta)^{-1} {E_Y^\delta}' B E_X^\delta w^\delta)(E_Y^\delta ({E_Y^\delta}' A_sE_Y^\delta)^{-1} {E_Y^\delta}' B E_X^\delta w^\delta)\\
&\hspace*{20em}+\beta  ({E_X^\delta}' \gamma_0' \gamma_0 E_X^\delta w^\delta)(w^\delta)\\
&=
\big(({E_{\new{X}}^\delta}' B' E_Y^\delta ({E_Y^\delta}' A_sE_Y^\delta)^{-1} {E_Y^\delta}' B E_X^\delta +\beta  {E_X^\delta}' \gamma_0' \gamma_0 E_X^\delta )w^\delta\big)(w^\delta).
\end{split}
\ee

Using \eqref{cond1}, we write $E_X^\delta=E_Y^\delta F^\delta$ with $F^\delta$ denoting the trivial embedding $X^\delta \rightarrow Y^\delta$.
Using $B =C +A_s$ and $C+C'+\gamma_0'\gamma_0=\gamma_T'\gamma_T$, similar to \eqref{BC} we infer that
\be \label{equality}
\begin{split}
&{E_X^\delta}' B'E_Y^\delta ({E_Y^\delta}' A_s E_Y^\delta)^{-1} {E_Y^\delta}' B E_X^\delta
+{E_X^\delta}' \beta \gamma_0' \gamma_0 E_X^\delta\\
&=
{F^\delta}'\Big({E_Y^\delta}' B' E_Y^\delta ({E_Y^\delta}' A_s E_Y^\delta)^{-1} {E_Y^\delta}' B E_Y^\delta
+{E_Y^\delta}'\beta  \gamma_0' \gamma_0 E_Y^\delta\Big)F^\delta\\
&={F^\delta}'\Big({E_Y^\delta}' C' E_Y^\delta ({E_Y^\delta}' A_s E_Y^\delta)^{-1} {E_Y^\delta}' C E_Y^\delta+{E_Y^\delta}' A_s E_Y^\delta+{E_Y^\delta}' (\gamma_T' \gamma_T +(\beta-1)\gamma_0' \gamma_0)E_Y^\delta\Big)F^\delta\\
&={E_X^\delta}' C' E_Y^\delta ({E_Y^\delta}' A_s E_Y^\delta)^{-1} {E_Y^\delta}' C E_X^\delta+{E_X^\delta}' A_s E_X^\delta
+{E_X^\delta}'  (\gamma_T' \gamma_T +(\beta-1)\gamma_0' \gamma_0) E_X^\delta.
\end{split}
\ee
We conclude that for any  $\eta \in (0,1]$,
\be \label{e8}
\begin{split}
\big((&{E_X^\delta}' B'E_Y^\delta ({E_Y^\delta}' A_s E_Y^\delta)^{-1} {E_Y^\delta}' B E_X^\delta
+{E_X^\delta}' \beta \gamma_0' \gamma_0 E_X^\delta)w^\delta\big)(w^\delta)\\ 
&= \|{E_Y^\delta}' C E_X^\delta w^\delta\|_{{Y^\delta}'}^2+\|E_X^\delta w^\delta\|_Y^2+\|\gamma_T E_X^\delta w^\delta\|^2+(\beta-1)\|\gamma_0 E_X^\delta w^\delta\|^2\\
& \geq (\|{E_Y^\delta}' \partial_t E_X^\delta w^\delta\|_{{Y^\delta}'}-\alpha\|E_X^\delta w^\delta\|_Y)^2+\|E_X^\delta w^\delta\|_Y^2+\|\gamma_T E_X^\delta w^\delta\|^2\\
& \hspace*{20em}
+(\beta-1)\|\gamma_0 E_X^\delta w^\delta\|^2\\
& \geq (1-\eta^2)\|{E_Y^\delta}' \partial_t E_X^\delta w^\delta\|_{{Y^\delta}'}^2+\big((1-\eta^{-2})\alpha^2+1\big)\|E_X^\delta w^\delta\|^2_Y+\|\gamma_T E_X^\delta w^\delta\|^2\\  &\hspace*{20em}+(\beta-1)\|\gamma_0 E_X^\delta w^\delta\|^2
\\
& \stackrel{\hspace*{-0.5em}\eqref{cond2}\hspace*{-1em}}{\geq} (1-\eta^2)(\gamma_\Delta^{\partial_t})^2  \|\partial_t E_X^\delta w^\delta\|_{Y'}^2+\big((1-\eta^{-2})\alpha^2+1\big)\|E_X^\delta w^\delta\|^2_Y+\|\gamma_T E_X^\delta w^\delta\|^2 \hspace*{-1em}\\
&\hspace*{20em}+(\beta-1)\|\gamma_0 E_X^\delta w^\delta\|^2\\
&\geq \min\Big((1-\eta^2)(\gamma_\Delta^{\partial_t})^2,\big((1-\eta^{-2})\alpha^2+1\big)\Big) \|E_X^\delta w^\delta\|_X^2,
\end{split}
\ee
where we applied Young's inequality. Solving $(1-\eta^2)(\gamma_\Delta^{\partial_t})^2=\big((1-\eta^{-2})\alpha^2+1\big)$ for $\eta$
yields
\be \label{e9}
(1-\eta^2)(\gamma_\Delta^{\partial_t})^2={\textstyle \frac12}\Big((\gamma_\Delta^{\partial_t})^2+\alpha^2+1-\sqrt{((\gamma_\Delta^{\partial_t})^2+\alpha^2+1)^2-4 (\gamma_\Delta^{\partial_t})^2}\Big)>0.
\ee
Recalling \eqref{e5} and $\|P^\delta\|_{\cL(X,X)}=\sup_{0 \neq w \in X} \frac{\|w^\delta\|_X}{\|w\|_X}$, the proof is completed by combining \eqref{e6}, \eqref{e7}, and \eqref{e8}.
\end{proof}

\section{Br\'{e}zis--Ekeland--Nayroles (BEN) formulation} \label{S4}
The minimizer $u \in X$ of $\Big\|\left[\begin{array}{@{}c@{}} B \\ \sqrt{\beta}\,\gamma_0\end{array} \right] w-\left[\begin{array}{@{}c@{}} g \\ \sqrt{\beta}\, u_0\end{array} \right]\Big\|^2_{Y' \times H}$, that is equal to the unique solution of \eqref{x12}, is the unique solution of
\be \label{100}
(B' A_s^{-1} B+\beta \gamma_0' \gamma_0)u= B' A_s^{-1}g+\beta \gamma_0' u_0.
\ee
As we have seen in \eqref{BC}, this system is equivalent to
\be \label{101}
(C' A_s^{-1} C + A_s +\gamma_T' \gamma_T+(\beta-1)\gamma_0' \gamma_0)u=(\identity+ C' A_s^{-1}) g + \beta \gamma_0' u_0,
\ee
showing that $u$ is the second component of the pair $(\lambda,u) \in Y \times X$ that solves
 \be \label{saddle}
 \left[\begin{array}{@{}cc@{}} A_s & C \\ C'  & -(A_s+\gamma_T' \gamma_T+(\beta-1)\gamma_0' \gamma_0) \end{array}\right]
 \left[\begin{array}{@{}c@{}} \lambda \\ u \end{array}\right]=
 \left[\begin{array}{@{}c@{}} g \\ -(g+\beta \gamma_0' u_0) \end{array}\right].
\ee
Notice that $\lambda=u$.

The formulation \eqref{101} of the parabolic equation can alternatively be derived from the application of the Br\'{e}zis--Ekeland--Nayroles variational principle (\cite{35.81,233.5}, cf.~also \cite[\S3.2.4]{10.6}), which generalizes beyond  the linear, Hilbert space setting.

Given $\delta \in \Delta$, we consider the Galerkin discretization of \eqref{saddle}, i.e.,
\be \label{BENsaddle}
\left[\begin{array}{@{}cc@{}} {E^\delta_Y}'A_s E^\delta_Y & {E^\delta_Y}'C E^\delta_X \\ ({E^\delta_Y}'C E^\delta_X)' & -{E^\delta_X}'(A_s+\gamma_T' \gamma_T+(\beta-1)\gamma_0' \gamma_0)E^\delta_X \end{array}\right]
\left[\begin{array}{@{}c@{}} \lambda^\delta \\ \bar{u}^\delta \end{array}\right]=
\left[\begin{array}{@{}c@{}}  {E^\delta_Y}' g\\  -{E^\delta_X}' (g+\beta \gamma_0' u_0)\end{array}\right]
\ee
or, equivalently
\be \label{BEN}
\begin{split}
{E_X^\delta}' \big(C' E_Y^\delta ({E_Y^\delta}' A_s E_Y^\delta)^{-1} {E_Y^\delta}' C
&+A_s+\gamma_T' \gamma_T+(\beta-1)\gamma_0' \gamma_0\big) E_X^\delta \bar{u}^\delta\\
&=
{E_X^\delta}' \big(C' E_Y^\delta ({E_Y^\delta}' A_s E_Y^\delta)^{-1} {E_Y^\delta}'g+g+\beta \gamma_0' u_0\big).
\end{split}
\ee

\begin{remark} \label{rem10} Assuming $X^\delta \subseteq Y^\delta$ (\eqref{cond1}) and $K_Y^\delta=({E_Y^\delta}' A_s E_Y^\delta)^{-1}$, it holds that  $\bar{u}^\delta=u^\delta$, i.e., the solutions of BEN  and MR \emph{are equal}. Indeed, \eqref{equality} shows that in this case the operator at the left-hand side of \eqref{BEN} equals the operator in \eqref{MRM}, and from
${E_X^\delta}' A_s E_Y^\delta ({E_Y^\delta}' A_s E_Y^\delta)^{-1} {E_Y^\delta}'={E_X^\delta}'$
when $X^\delta \subseteq Y^\delta$ one deduces that also the right-hand sides agree.

In contrast to MR, with BEN, however, it is \emph{not} possible to replace $({E_Y^\delta}' A_s E_Y^\delta)^{-1}$ by a general preconditioner as in \eqref{m9}-\eqref{MRM} and still obtain a quasi-best approximation to $(\lambda,u)$ from $Y^\delta \times X^\delta$. This can be understood by noticing that
replacing $A_s^{-1}$ in \eqref{101} by another operator changes the solution, whereas this is not the case in \eqref{100}.
So for the iterative solution of BEN one has to operate on the saddle point system \eqref{BENsaddle} instead of on a symmetric positive definite system as with MR, see \eqref{MRM}.

On the other hand, with BEN it is not needed that $X^\delta \subseteq Y^\delta$, as we will see below.
\end{remark}

The applicability of BEN for the case that $A_a \neq 0$ was already demonstrated in \cite{249.99}. The following result
gives a quantitatively better error bound.

\begin{theorem} \label{thm:mainBEN} Under the sole condition \eqref{cond2}, the solution $\bar{u}^\delta \in X^\delta$ of \eqref{BEN} exists uniquely, and satisfies
$$
\|u-\bar{u}^\delta\|_X \leq {\textstyle \frac{\big(1+{\textstyle \frac12}\big(\alpha^2+\alpha\sqrt{\alpha^2+4}\big)\big) {\displaystyle \inf_{w \in X^\delta} \|u-w\|_X}+\sqrt{1+\alpha^2}\,{\displaystyle \inf_{v \in Y^\delta} \|u-v\|_Y}}{{\textstyle \frac12}\big((\gamma_\Delta^{\partial_t})^2+\alpha^2+1-\sqrt{((\gamma_\Delta^{\partial_t})^2+\alpha^2+1)^2-4 (\gamma_\Delta^{\partial_t})^2}\big)}}.
$$
\end{theorem}

\begin{proof}
With $g=Bu$ and $u_0=\gamma_0 u$, using $B=C+A_s$ and $\gamma_0' \gamma_0=\gamma_T' \gamma_T-(C'+C)$, the right-hand side of \eqref{BEN} reads as
\begin{align*}
&{E_X^\delta}' \big(C' E_Y^\delta ({E_Y^\delta}' A_s E_Y^\delta)^{-1} {E_Y^\delta}' (C+A_s)+A_s+\gamma_T' \gamma_T+(\beta-1)\gamma_0' \gamma_0-C'\big)u=
\\
&{E_X^\delta}' \big( C' E_Y^\delta ({E_Y^\delta}' A_s E_Y^\delta)^{-1} {E_Y^\delta}' C +A_s+\gamma_T' \gamma_T+(\beta-1)\gamma_0' \gamma_0
\\
&\hspace*{15em}+C' \big[E_Y^\delta ({E_Y^\delta}' A_s E_Y^\delta)^{-1} {E_Y^\delta}' A_s -\identity\big]\big)u.
\end{align*}
So with $G(\delta):=C' E_Y^\delta ({E_Y^\delta}' A_s E_Y^\delta)^{-1} {E_Y^\delta}' C
+A_s+\gamma_T' \gamma_T+(\beta-1)\gamma_0' \gamma_0$, it holds that
$$
u \mapsto E_X^\delta\bar{u}^\delta=E_X^\delta({E_X^\delta}' G(\delta) E_X^\delta)^{-1}{E_X^\delta}' \big( G(\delta)+C' \big[E_Y^\delta ({E_Y^\delta}' A_s E_Y^\delta)^{-1} {E_Y^\delta}' A_s -\identity\big]\big)u,
$$
where we already used that ${E_X^\delta}' G(\delta) E_X^\delta$ is invertible, which will be verified below.
Since $E_X^\delta({E_X^\delta}' G(\delta) E_X^\delta)^{-1}{E_X^\delta}'  G(\delta) \in \cL(X,X)$ and $E_Y^\delta ({E_Y^\delta}' A_s E_Y^\delta)^{-1} {E_Y^\delta}' A_s \in \cL(Y,Y)$ are projectors onto $X^\delta$ and $Y^\delta$, respectively, the latter being orthogonal, for any $v \in Y^\delta$ and $w \in X^\delta$ it holds that
\begin{align*}
u-\bar{u}^\delta=&(\identity-E_X^\delta({E_X^\delta}' G(\delta) E_X^\delta)^{-1}{E_X^\delta}'  G(\delta))(u-E_X^\delta w)\\
&+E_X^\delta({E_X^\delta}' G(\delta) E_X^\delta)^{-1} {E_X^\delta}' C' \big[\identity-E_Y^\delta ({E_Y^\delta}' A_s E_Y^\delta)^{-1} {E_Y^\delta}' A_s\big](u-E_Y^\delta v)
\end{align*}
and so, also using $Y^\delta \not\in\{0,Y\}$,
\begin{align*}
\|u-\bar{u}^\delta\|_X \leq
\|({E_X^\delta}' G(\delta) E_X^\delta)^{-1}\|_{\cL({X^\delta}',X^\delta)}\Big\{&\|G(\delta)\|_{\cL(X,X')} \inf_{w \in X^\delta}\|u-w\|_X\\&+
\|C\|_{\cL(X,Y')} \inf_{v \in Y^\delta}\|u-v\|_Y
\Big\}.
\end{align*}

For $w \in X$, we have
\begin{align*}
(G(\delta)w)(w) & = \|{E_Y^\delta}' Cw\|^2_{{Y^\delta}'}+\|w\|_Y^2+\|\gamma_T w\|^2+(\beta-1)\|\gamma_0 w\|^2
\\ &\leq \|Cw\|^2_{Y'}+\|w\|_Y^2+\|\gamma_T w\|^2 +(\beta-1)\|\gamma_0 w\|^2
\\
& =((C' A_s^{-1} C + A_s+\gamma_T' \gamma_T+(\beta-1)\gamma_0' \gamma_0)w)(w)
=\|Bw\|^2_{Y'}+\beta \|\gamma_0 w\|^2 \\
&\leq \Big(1+{\textstyle \frac12}\Big(\alpha^2+\alpha\sqrt{\alpha^2+4}\Big)\Big) \|w\|_X^2
\end{align*}
  by Proposition~\ref{prop:continuous}. Since $(G(\delta)\cdot)(\cdot)$ is symmetric semi-positive-definite, we conclude that $\|G(\delta)\|_{\cL(X,X')} \leq 1+{\textstyle \frac12}\Big(\alpha^2+\alpha\sqrt{\alpha^2+4}\Big)$.

For $w \in X^\delta$, one deduces
\begin{align*}
(G(\delta)E_X^\delta w)(E_X^\delta w)&= \|{E_Y^\delta}' CE_X^\delta w\|^2_{{Y^\delta}'}+\|E_X^\delta w\|_Y^2+\|\gamma_T E_X^\delta w\|^2+(\beta-1)\|\gamma_0 E_X^\delta w\|^2\\
&\geq {\textstyle \frac12}\Big((\gamma_\Delta^{\partial_t})^2+\alpha^2+1-\sqrt{((\gamma_\Delta^{\partial_t})^2+\alpha^2+1)^2-4 (\gamma_\Delta^{\partial_t})^2}\Big) \|E_X^\delta w\|_X^2
\end{align*}
by following the lines starting at the second line of \eqref{e8}, in particular showing that ${E_X^\delta}'G(\delta)E_X^\delta$ is invertible.

Finally, for $w \in X$, $\|Cw\|_{Y'}\leq \|\partial_t w\|_{Y'}+\alpha \|w\|_Y \leq \sqrt{1+\alpha^2}\,\|w\|_X$.
The theorem follows by combining the above estimates.
\end{proof}

\section{Inf-sup condition \eqref{cond2}, i.e., $\gamma_\Delta^{\partial_t}>0$, and condition \eqref{cond3}} \label{S5}
By the boundedness and coercivity assumptions \eqref{boundedness} and \eqref{coercivity}, it holds that
$\|\cdot\|_Y \eqsim  \|\cdot\|_{L_2(I;V)}$. Since with
\be \label{37}
\gamma^\delta:=\gamma^\delta(X^\delta,Y^\delta):=\inf_{\{w \in X^\delta \colon \partial_t w \neq 0\}} \sup_{0 \neq v \in Y^\delta}
  \frac{\int_I \langle \partial_t w, v\rangle \dif t}{\|\partial_t w\|_{L_2(I;V')}\|v\|_{L_2(I;V)}},
\ee
consequently it holds that
$\gamma_\Delta^{\partial_t} \eqsim \,\inf_{\delta \in \Delta} \gamma^\delta$,
  we will summarize some known results about settings for which $\inf_{\delta \in \Delta} \gamma^\delta>0$  has been demonstrated.

In the final subsection of this section we will briefly comment on the construction of preconditioners at the $Y$-side, i.e.~condition \eqref{cond3}, and the $X$-side.
  The preconditioner $K_Y^\delta$ has its application for the reduction of the saddle-point system \eqref{m9} (reading $(K_Y^\delta)^{-1}$ as ${E_Y^\delta}'A_sE_Y^\delta$) to the elliptic system \eqref{MRM}, and as an ingredient for building a preconditioner for the saddle-point system \eqref{BENsaddle}, whereas $K_X^\delta$  can be applied for preconditioning \eqref{MRM}, and as the other ingredient to construct a preconditioner for \eqref{BENsaddle}.
    \medskip

\new{%
  Since inf-sup or Ladyzhenskaya--Babu\v{s}ka--Brezzi (LBB) conditions of type \mbox{$\gamma^\delta>0$} will be encountered often, in an abstract framework  in the following Proposition~\ref{39} we establish their relation to existence of a Fortin operator, denoted by $Q$. Since the work of Fortin (\cite{75.22}), it is well-known that existence of such an operator implies the LBB condition.
 We show that also the converse is true, and present a quantitatively optimal statement.
Moreover, in contrast to the common presentation (although not in \cite{75.22}), in view of applications the operator $F$  in Proposition~\ref{39} is not required to be injective.
The  estimates from \cite[Lemma 26.9]{70.98}, which apply under the `continuous' inf-sup condition $\inf_{0 \neq a \in \mathscr{A}} \frac{\|Fa\|_{\mathscr{B}'}}{\|a\|_{\mathscr{A}}}>0$, are in that case similar to those from Proposition~\ref{39}, and can easily be derived from this result.

\begin{proposition} \label{39} For Hilbert spaces $\mathscr{A}$ and $\mathscr{B}$, let $F \in \cL(\mathscr{A},\mathscr{B}')$.
Let $\mathfrak{A} \subset \mathscr{A}$ and $\mathfrak{B}  \subset \mathscr{B}$ be closed subspaces with $F\mathfrak{A} \neq \{0\}$ and $\mathfrak{B} \neq \{0\}$. Let  $E_{\mathfrak{A}} \colon \mathfrak{A} \rightarrow \mathscr{A}$ and $E_{\mathfrak{B}} \colon \mathfrak{B} \rightarrow \mathscr{B}$ denote the trivial embeddings, which we sometimes write for clarity, but that we mainly introduce for their duals.
If there exists a
\be \label{38}
  Q \in \cL(\mathscr{B},\mathscr{B}) \text{ with } \ran Q \subset \mathfrak{B} \text{ and } (F  \mathfrak{A})((\identity-Q)\mathscr{B})=0,
\ee
then $\mathfrak{G}:=\inf_{\{\mathfrak{a} \in \mathfrak{A} \colon F  \mathfrak{a} \neq 0\}} \frac{\|E_{\mathfrak{B}}' F E_{\mathfrak{A}}\mathfrak{a}\|_{\mathfrak{B}'}}{\|F \mathfrak{a}\|_{\mathscr{B}'}} \geq \|Q\|_{\cL(\mathscr{B},\mathscr{B})}^{-1}$. Conversely, if \mbox{$\mathfrak{G}>0$,} and $\ran E_{\mathfrak{B}}' F E_{\mathfrak{A}}$ is closed, then
   then a $Q$ as in \eqref{38} exists, which moreover is a projector, with  $\|Q\|_{\cL(\mathscr{B},\mathscr{B})}= 1/\mathfrak{G}$.
   The condition of the closedness of $\ran E_{\mathfrak{B}}' F E_{\mathfrak{A}}$ can be replaced by $\dim \mathfrak{A}<\infty$, or by the closedness  of $\ran F$.
\end{proposition}

\begin{proof} This proof resembles that of  \cite[Thm.~3.11]{58.6}, but yields quantitatively optimal bounds.

If a $Q$ as in \eqref{38} exists, then for $\mathfrak{a} \in \mathfrak{A}$ it holds that
$$
\|F\mathfrak{a}\|_{\mathscr{B}'}=\sup_{0 \neq \beta \in \mathscr{B}}\frac{(F \mathfrak{a})(\beta)}{\|\beta\|_\mathscr{B}}=\sup_{0 \neq \beta \in \mathscr{B}}\frac{(F\mathfrak{a})(Q \beta)}{\|\beta\|_\mathscr{B}}
\leq \|Q\|_{\cL(\mathscr{B},\mathscr{B})} \sup_{0 \neq \mathfrak{b} \in \mathfrak{B}}\frac{(F\mathfrak{a})(\mathfrak{b})}{\|\mathfrak{b}\|_\mathscr{B}},
$$
or $\mathfrak{G} \geq \|Q\|_{\cL(\mathscr{B},\mathscr{B})}^{-1}$.

Now let $\mathfrak{G}>0$.
By the open mapping, the  closedness  of $\ran F$ is equivalent to $\|F[\alpha]\|_{\mathscr{B}'} \eqsim \|[\alpha]\|_{\mathscr{A}/\ker F}$ ($[\alpha] \in \mathscr{A}/\ker F$). Thanks to $\mathfrak{G}>0$, the latter implies
\be \label{infsupje}
\|E_{\mathfrak{B}}' FE_{\mathfrak{A}}[\mathfrak{a}]\|_{\mathfrak{B}'} \eqsim \|[\mathfrak{a}]\|_{\mathscr{A}/\ker F} \quad ([\mathfrak{a}] \in \mathfrak{A}/\ker F),
\ee
which in turn is equivalent to the closedness  of $\ran E_{\mathfrak{B}}' FE_{\mathfrak{A}}$. Obviously, the latter holds also true when $\dim \mathfrak{A}<\infty$.

With the Riesz map $R\colon \mathscr{B} \rightarrow \mathscr{B}'$, we define $Q\colon \mathscr{B} \rightarrow \mathfrak{B} \colon \beta \mapsto \mathfrak{b}$  with the latter being the first component\footnote{One may verify that $\mathfrak{b}=\argmin_{\{\tilde{\mathfrak{b}}\colon (F\mathfrak{A})(\beta-\tilde{\mathfrak{b}})=0\}} \|\tilde{\mathfrak{b}}\|_{\mathscr{B}}$.} of $(\mathfrak{b},[\mathfrak{a}])\in \mathfrak{B} \times \mathfrak{A}/\ker F$ that solves
$$
\left[\begin{array}{@{}cc@{}}E_{\mathfrak{B}}' RE_{\mathfrak{B}} &E_{\mathfrak{B}}' FE_{\mathfrak{A}}\\
E_{\mathfrak{A}}' F'E_{\mathfrak{B}}& 0 \end{array}\right]
\left[\begin{array}{@{}c@{}} \mathfrak{b}\\ {[\mathfrak{a}]} \end{array}\right]
=
\left[\begin{array}{@{}c@{}} 0 \\E_{\mathfrak{A}}' F' \beta \end{array}\right].
$$
We will see that this system is uniquely solvable.

We equip $\mathfrak{A}/\!\ker F$ with norm $\|E_{\mathfrak{B}}' FE_{\mathfrak{A}} \cdot\|_{{\mathfrak{B}}'}$.
Thanks to \eqref{infsupje}, with this norm and corresponding scalar product, $\mathfrak{A}/\!\ker F$ is a Hilbert space, which implies the surjectivity of the corresponding Riesz map.

One verifies that both
$E_{\mathfrak{B}}' RE_{\mathfrak{B}}\colon \mathfrak{B} \rightarrow {\mathfrak{B}}'$ and the Schur complement
$S := E_{\mathfrak{A}}' F'E_{\mathfrak{B}} (E_{\mathfrak{B}}' RE_{\mathfrak{B}})^{-1}E_{\mathfrak{B}}' FE_{\mathfrak{A}} \colon \mathfrak{A}/\!\ker F \rightarrow (\mathfrak{A}/\!\ker F)'$ are Riesz maps.
Using $S [\mathfrak{a}] = E_{\mathfrak A}' F' \beta$, we infer that
$$
\|\mathfrak{b}\|_{\mathscr{B}} = \|E_{\mathfrak{B}}' FE_{\mathfrak{A}} [\mathfrak{a}]\|_{{\mathfrak{B}}'} = \|[\mathfrak{a}]\|_{\mathfrak{A}/\!\ker F} =\|E_{\mathfrak{A}}' F' \beta\|_{(\mathfrak{A}/\!\ker F)'}.
$$
From
\begin{align*}
\|E_{\mathfrak{A}}' F'\|_{\cL(\mathscr{B},(\mathfrak{A}/\!\ker F)')}&
  =\|F E_{\mathfrak{A}}\|_{\cL(\mathfrak{A}/\!\ker F,\mathscr{B}')}
  \\
&=\sup_{\{\mathfrak{a} \in \mathfrak{A}\colon F \mathfrak{a} \neq 0\}} \inf_{0 \neq \mathfrak{b} \in \mathfrak{B}} \frac{\|F \mathfrak{a}\|_{\mathscr{B}'} \|\mathfrak{b}\|_\mathscr{B}}{(F \mathfrak{a})(\mathfrak{b})} = 1/\mathfrak{G},
\end{align*}
we conclude that $\|Q\|_{\cL(\mathscr{B},\mathscr{B})} = 1/\mathfrak{G}$, which completes the proof.
\end{proof}%
}

\subsection{`Full' tensor product case} \label{full}
Concerning the verification of $\inf_{\delta \in \Delta} \gamma^\delta>0$, we start with the easy case of $X^\delta$ and $Y^\delta$ being \emph{`full' tensor products} of approximation spaces in time and space (as opposed to sparse tensor products, see below).
With $Y_t:=L_2(I)$ and $X_t:=H^1(I)$, for $Z\in \{X,Y\}$ let $(Z^\delta_t)_{\delta \in \Delta}$ and $(Z^\delta_{\bf x})_{\delta \in \Delta}$ be families of closed subspaces of $Z_t$ and $V$, respectively, and let $Z^\delta:=Z_t^\delta \otimes Z_{\bf x}^\delta$.
Assuming that
\begin{align} \label{subt}
\gamma^\delta_t &:= \inf_{\{w \in X^\delta_t \colon w' \neq 0\}}
\sup_{0 \neq v \in Y^\delta_t} \frac{\int_I w' v \dif t}{\|w'\|_{L_2(I)}\|v\|_{L_2(I)}} \gtrsim 1,\\ \label{subx}
  \gamma^\delta_{\bf x} &:= \inf_{0 \neq w \in X^\delta_{\bf x}}
  \sup_{0 \neq v \in Y^\delta_{\bf x}} \frac{\langle w, v \rangle }{\|w\|_{V'} \|v\|_V } \gtrsim 1,
\end{align}
a tensor product argument shows that
$$
\gamma^{\delta}= \gamma^\delta_t \gamma^\delta_{\bf x} \gtrsim 1.
$$

Obviously, \eqref{subt} is true when $\frac{\dif}{\dif t} X^\delta_t \subseteq Y^\delta_t$, which however is not a necessary condition. For example, when $X^\delta_t$ is the space of continuous piecewise linears w.r.t.~some partition of $I$, and
$Y^\delta_t$ is the space of continuous piecewise linears w.r.t.~a once dyadically refined partition, an easy computation  (\cite[Prop.~6.1]{11}) shows that $\gamma^\delta_t \geq \sqrt{3/4}$.

Considering, for a domain $\Omega \subset \R^d$ and $\Gamma \subset \partial\Omega$,  $H=L_2(\Omega)$ and $V=H^1_{0,\Gamma}(\Omega):=\{v \in H^1(\Omega)\colon v|_\Gamma=0\}$, $H^1(\Omega)$-stability of the $L_2(\Omega)$-orthogonal projector onto Lagrange finite element spaces $X^\delta_{\bf x}=Y^\delta_{\bf x}$ is an extensively studied subject. In view of Proposition~\ref{39}, taking $F$ to be the Riesz map $H \rightarrow H'$ viewed as a mapping $V \rightarrow V'$, this stability implies \eqref{subx}.
For finite element spaces w.r.t.~shape regular quasi-uniform partitions into, say, $d$-simplices, where $\Gamma$ is the union of faces of $T \in \tria$, this stability follows easily from direct and inverse estimates. It is known that this stability holds also true for (shape regular) locally refined partitions when they are sufficiently mildly graded. In \cite{75.3675}, it is shown that in two space dimensions the meshes generated by newest vertex bisection satisfy this requirement, see also \cite{64.595} for extensions.

\subsection{Sparse tensor product case}
As shown in \cite[Prop.~4.2]{11}, these results for full tensor products extend to \emph{sparse tensor products}.
When $(Z^\delta_t)_{\delta \in \Delta}$ and $(Z^\delta_{\bf x})_{\delta \in \Delta}$ are nested sequences of closed subspaces
$Z_t^{\delta_0} \subset Z_t^{\delta_1} \subset \cdots \subset Z_t$, $Z_{\bf x}^{\delta_0} \subset Z_{\bf x}^{\delta_1} \subset \cdots \subset V$ that satisfy \eqref{subt}--\eqref{subx}, then for $Z^{\delta_n}:=\sum_{\{0 \leq n_t+n_{\bf x}\leq n\}} Z_t^{\delta_{n_t}} \otimes Z_{\bf x}^{\delta_{n_{\bf x}}}$ it holds that
$$
\gamma^{\delta_n} \geq \min_{0 \leq n_t \leq n} \gamma_t^{\delta_{n_t}}\min_{0 \leq n_{\bf x} \leq n}\gamma_t^{\delta_{n_{\bf x}}} \gtrsim 1.
$$
\medskip

\subsection{Time-slab partition case}
Another extension of the full tensor product case is given by the following.
Let $(\bar{X}^\delta, \bar{Y}^\delta)_{\delta \in \bar{\Delta}}$ be a family of pairs of closed subspaces of $X$ and $Y$
for which
$$
\gamma_{\bar{\Delta}}:=\inf_{\delta \in \bar{\Delta}}\inf_{\{w \in \bar{X}^\delta \colon \partial_t w \neq 0\}} \sup_{0 \neq v \in \bar{Y}^\delta}
  \frac{\int_I \langle \partial_t w, v\rangle \dif t}{\|w\|_{L_2(I;V')}\|v\|_{L_2(I;V)}}>0.
  $$
Then if, for $\delta \in \Delta$, $X^\delta$ and $Y^\delta$ are such that for some
finite partition $I^\delta=([t_{i-1}^\delta,t_{i}^\delta])_i$ of $I$, with
$G^\delta_i(t):=t_{i-1}^\delta+\frac{t}{T}(t_{i}^\delta-t_{i-1}^\delta)$ and arbitrary $\delta_i \in \bar{\Delta}$ it holds that
\begin{align*}
X^\delta \subseteq \{u \in X &\colon u|_{(t^\delta_{i-1},t^\delta_i)} \circ G^\delta_i \in \bar{X}^{\delta_i}\},\\
Y^\delta \supseteq \{v \in L_2(I;V) &\colon v|_{(t^\delta_{i-1},t^\delta_i)} \circ G^\delta_i \in \bar{Y}^{\delta_i}\},
\end{align*}
then $\gamma^\delta \geq \gamma_{\bar{\Delta}}>0$ as one easily verifies by writing $\int_I  \langle \frac{\dif u}{\dif t}, v\rangle \dif t=\sum_i \int_{t_{i-1}}^{t_i} \langle \frac{\dif u}{\dif t}, v\rangle \dif t$. An example of this \emph{`time-slab partition'} setting will be given in Sect.~\ref{Sconcrete}. Thinking of the $\bar{X}^\delta$ as being finite element spaces, notice that the condition $X^\delta \subset X$ will require that possible `hanging nodes' on the interface between different time slabs do not carry degrees of freedom.
\medskip

\subsection{Generalized sparse tensor product case}
Finally, we informally describe a \emph{`generalized' sparse tensor product} setting that allows for \emph{local refinements} driven by an a posteriori error estimator.
For $Z \in \{X,Y\}$, let the nested sequences of closed subspaces $Z_t^{\delta_0} \subset Z_t^{\delta_1} \subset \cdots \subset Z_t$, $Z_{\bf x}^{\delta_0} \subset Z_{\bf x}^{\delta_1} \subset \cdots \subset V$ be equipped with hierarchical bases, meaning that the basis for $Z_t^{\delta_i}$ (analogously $Z_{\bf x}^{\delta_i}$) is inductively defined as the basis for $Z_t^{\delta_{i-1}}$ plus a basis for a complement space of $Z_t^{\delta_{i-1}}$ in $Z_t^{\delta_{i}}$. The \emph{level} of the functions in the latter basis is defined as $i$.

Let us consider the usual case that the diameter of the support of a hierarchical basis function with level $i$ is $\eqsim 2^{-i}$, and let us assign to each basis function $\phi$ on level $i>0$ one (or a few) parents with level $i-1$ whose supports intersect the support of $\phi$.
We now let $(Z^\delta)_{\delta \in \Delta}$ be the collection of all spaces that are spanned by sets of product hierarchical basis functions, which sets are \emph{downward closed} (or \emph{lower}) in the sense that if a product of basis functions
is in the set, then so are all their parents in both directions.
Note that the sparse tensor product spaces  $\sum_{\{0 \leq n_t+n_{\bf x}\leq n\}} Z_t^{\delta_{n_t}} \otimes Z_{\bf x}^{\delta_{n_{\bf x}}}$ are included in this collection, but that it contains many more spaces.

Under conditions on the hierarchical bases for $Z_t^{\delta_0} \subset Z_t^{\delta_1} \subset \cdots \subset Z_t$ for $Z \in \{X,Y\}$, that should be of \emph{wavelet-type}, in \cite{249.991} it is shown that to any $X^\delta$ one can assign a
 $Y^\delta$ with $\dim Y^\delta \lesssim \dim X^\delta$, such that $\gamma^\delta \gtrsim 1$ holds.

\subsection{Preconditioners}
Moving to condition \eqref{cond3}, obviously we would like to construct $K_Y^\delta$ such that it is not only a uniform preconditioner, i.e., it satisfies \eqref{cond3}, but also that its application can be performed in ${\mathcal O}(\dim Y^\delta)$ operations. In the full-tensor product case, after selecting bases for $Y^\delta_t$ and $Y^\delta_{\bf x}$, the construction of
$K_Y^\delta$ boils down to tensorizing approximate inverses of the `mass matrix' in time, which does not pose any problems, and the `stiffness matrix' in space.
For $V=H^1(\Omega)$ (or a subspace of aforementioned type), it is well-known that by taking a multi-grid preconditioner
as the approximate inverse of the stiffness matrix the resulting $K_Y^\delta$ satisfies our needs.
A straightforward generalization of this construction of $K_Y^\delta$ applies to spaces $Y^\delta$ that correspond to the time-slab partitioning approach.
\medskip

Finally, for the efficient iterative solution of \eqref{MRM} or \eqref{BENsaddle}, one needs a $K_X^\delta={K_X^\delta}' \in
\Lis({X^\delta}',X^\delta)$ whose norm and norm of its inverse are uniformly bounded, and whose application can be performed in ${\mathcal O}(\dim X^\delta)$ operations.
For the full/sparse and generalized sparse tensor product setting such preconditioners have been constructed in \cite{12.5} and \cite{249.991}, respectively.

\section{Robustness} \label{S6}
The quasi-optimality results presented in Theorems~\ref{thm:mainMRM} and \ref{thm:mainBEN} for MR and BEN degenerate when $\alpha=\|A_a\|_{\cL(Y,Y')} \rightarrow \infty$. Aiming at results that are robust for $\alpha \rightarrow \infty$, we now study convergence w.r.t.~the energy-norm $\nrm\cdot\nrm_X$ on $X$. On its own this change of norms turns out not to be helpful. By replacing $\|\cdot\|_X$ by $\nrm\cdot\nrm_X$ in Theorems~\ref{thm:mainMRM} and \ref{thm:mainBEN}, and adapting their proofs in an obvious way yields for MR the same upper bound for $\frac{\nrm u-u^\delta \nrm_X}{\inf_{w \in X^\delta} \nrm u - w \nrm_X}$ as we found for $\frac{\|u-u^\delta\|_X}{\inf_{w \in X^\delta} \| u - w \|_X}$ (for $u \not\in X^\delta$),
whereas instead of Theorem~\ref{thm:mainBEN} we arrive at the only slightly more favourable bound
$$
\nrm u-\bar{u}^\delta\nrm_X \leq
{\textstyle
\frac{2+\alpha^2+\alpha\sqrt{\alpha^2+4}}
{(\gamma_\Delta^{\partial_t})^2+\alpha^2+1-\sqrt{((\gamma_\Delta^{\partial_t})^2+\alpha^2+1)^2-4 (\gamma_\Delta^{\partial_t})^2}}}
\inf_{w \in X^\delta,\,v \in Y^\delta} \nrm u-w\nrm_X+ \|u-v\|_Y,
$$
which is, however, still far from being robust.

In order to obtain robust bounds, instead of the condition $\gamma_\Delta^{\partial_t}>0$ (\eqref{cond2}) we now impose
\be \label{cond2bis}
\gamma_\Delta^C:=\inf_{\delta \in \Delta} \inf_{\{0 \neq w \in X^\delta\colon C E_X^\delta w \neq 0\}} \frac{\|{E_Y^\delta}'C E_X^\delta w\|_{{Y^\delta}'}}{\| C E_X^\delta w\|_{Y'}} >0,
\ee
which, when considering a family of operators $A$, we would like to hold uniformly for $\alpha  \rightarrow \infty$.

\begin{theorem} \label{thm13} Under conditions \eqref{cond1}, \eqref{cond2bis}, and \eqref{cond3}, the solution $u^\delta \in X^\delta$ of \eqref{MRM} satisfies
\begin{align} \label{bndMR}
\nrm u-u^\delta\nrm_X &\leq
\sqrt{\tfrac{\max(R_\Delta,1)}{\min(r_\Delta,1)}}\, (\gamma_\Delta^C)^{-1}
 \inf_{w \in X^\delta} \nrm u-w\nrm_X;
\intertext{and under condition \eqref{cond2bis}, the  solution $\bar{u}^\delta \in X^\delta$ of \eqref{BEN} satisfies} \label{bndBEN}
\nrm u-\bar{u}^\delta\nrm_X &\leq (\gamma_\Delta^C)^{-2}\big\{
\inf_{w \in X^\delta} \nrm u-w\nrm_X+ \inf_{v \in Y^\delta} \|u-v\|_Y\big\}.
\end{align}
\end{theorem}

\begin{proof}
The first estimate follows from ignoring the last inequality in \eqref{e6}, and by replacing the first inequality in \eqref{e8} by
\begin{align*}
&\|{E_Y^\delta}' C E_X^\delta w^\delta\|_{{Y^\delta}'}^2+\|E_X^\delta w^\delta\|_Y^2+\|\gamma_T E_X^\delta w^\delta\|^2+(\beta-1)\|\gamma_0 E_X^\delta w^\delta\|^2\\
& \geq (\gamma_\Delta^C)^2 \Big( \|C E_X^\delta w^\delta\|_{Y'}^2+\|E_X^\delta w^\delta\|_Y^2+\|\gamma_T E_X^\delta w^\delta\|^2+(\beta-1)\|\gamma_0 E_X^\delta w^\delta\|^2\Big)\\
&=
(\gamma_\Delta^C)^2
\big(({E_X^\delta}' B' A_s^{-1}  B E_X^\delta
+{E_X^\delta}' \beta \gamma_0' \gamma_0 E_X^\delta)w^\delta\big)(w^\delta)
=(\gamma_\Delta^C)^2 \nrm w^\delta\nrm_X^2.
\end{align*}

By following the proof of Theorem~\ref{thm:mainBEN}, recalling that  now $X$ is equipped with $\nrm\cdot\nrm_X$, from $\|C\|_{\cL(X,Y')} \leq 1$,
$\|G(\delta)\|_{\cL(X,X')} \leq 1$, and
$\|({E_X^\delta}' G(\delta) E_X^\delta)^{-1}\|_{\cL({X^\delta}',X^\delta)} \leq (\gamma_\Delta^C)^{-2}$, one infers the estimate for BEN.
\end{proof}

We conclude that for a family of (asymmetric) operators $A$ robustness w.r.t.~\mbox{$\nrm\cdot\nrm_X$} is obtained when $(\gamma_\Delta^C)^{-1}$ is uniformly bounded for $\alpha=\|A_a\|_{\cL(Y,Y')}  \rightarrow \infty$.
A family for which this will be realized is presented in Sect.~\ref{Sconcrete}.

\subsection{A posteriori error estimation}
In particular because for $\alpha=\|A_a\|_{\cL(Y,Y')}  \rightarrow \infty$ meaningful a priori error bounds for $\inf_{w \in X^\delta} \nrm u-w\nrm_X$ will be hard to derive, it is important to have (robust) a posteriori error bounds.

Let $Q_B^\delta \in \cL(Y,Y)$ be such that $\ran Q_B^\delta \subset Y^\delta$ and $(\identity-{Q_B^\delta}') B X^\delta=0$. Then,
with $e_{\rm osc}^\delta(g):=\|(\identity -{Q_B^\delta}')g\|_{Y'}$, for $w \in X^\delta$ and $u$ the solution of \eqref{x12} it holds that

\begin{align*}
r_\Delta\|{E_Y^\delta}' (g-&Bw)\|_{K_Y^\delta}^2+\beta\|u_0-\gamma_0 w\|^2
\leq \nrm u-w\nrm_X^2
\leq
\\ &\big(\|Q_B^\delta\|_{\cL(Y,Y)} \sqrt{R_\Delta} \|{E_Y^\delta}'(g-Bw)\|_{K_Y^\delta}+e_{\rm osc}^\delta(g)\big)^2+\beta\|u_0-\gamma_0 w\|^2,
\end{align*}
which follows from $\|g-Bw\|_{{Y^\delta}'} \leq \|g-Bw\|_{Y'} \leq \|{Q_B^\delta}'(g-Bw)\|_{Y'}+e_{\rm osc}^\delta(g)$.

We infer that if $\sup_{\delta \in \Delta} \|Q_B^\delta\|_{\cL(Y,Y)} <\infty$, then the a posteriori error estimator
\be\label{eqn:apost}
\cE^\delta(w; g,u_0,\beta):=\sqrt{\|{E_Y^\delta}'(g-Bw)\|^2_{K_Y^\delta}+\beta\|u_0-\gamma_0 w\|^2}
\ee
is an efficient and, modulo the \emph{data oscillation term} $e_{\rm osc}^\delta(g)$, reliable estimator  of the error $\nrm u-w\nrm_X$.
If $\sup_{\delta \in \Delta}\|Q_B^\delta\|_{\cL(Y,Y)}$ and $\frac{\max(R_\Delta,1)}{\min(r_\Delta,1)}$ are bounded uniformly in $\alpha \rightarrow \infty$,
then this estimator is even robust.

\begin{remark}
In view of Proposition~\ref{39}, the aforementioned assumptions $\ran Q_B^\delta \subset Y^\delta$, $(\identity-{Q_B^\delta}') B X^\delta=0$, and $\sup_{\delta \in \Delta} \|Q_B^\delta\|_{\cL(Y,Y)} <\infty$ are equivalent to
$$
\gamma_\Delta^B:=\inf_{\delta \in \Delta} \inf_{\{0 \neq w \in X^\delta\colon B E_X^\delta w \neq 0\}} \frac{\|{E_Y^\delta}'B E_X^\delta w\|_{{Y^\delta}'}}{\| B E_X^\delta w\|_{Y'}} >0.
$$
  In applications the conditions $\gamma_\Delta^{\partial_t}>0$, $\gamma_\Delta^C>0$, and $\gamma_\Delta^B>0$ are increasingly more difficult to fulfill.
\end{remark}

To have a meaningful reliability result, in addition we would like to find above $Q_B^\delta$ such that, for sufficiently smooth $g$, the term $e_{\rm osc}^\delta(g)$ is asymptotically, i.e.~for  the `mesh-size' tending to zero, of equal or higher order than the approximation error $\inf_{w \in X^\delta}\nrm u-w\nrm_X$. We will realize this in the setting that will be discussed in Sect.~\ref{Sapost}.

\section{Spatial differential operators with dominating asymmetric part} \label{Sconcrete}
For some domain $\Omega \subset \R^d$, and $\Gamma \subset \partial\Omega$, let
\be \label{bil}
\begin{split}
& H:=L_2(\Omega),\,V:=H^1_{0,\Gamma}(\Omega):=\{v \in H^1(\Omega)\colon v|_{\Gamma}=0\},\\
& a(t;\eta,\zeta):=\int_\Omega \eps \nabla \eta \cdot \nabla \zeta+ ({\bf b}\cdot \nabla \eta + e \eta) \zeta \dif{\bf x}, \quad\eps>0,\\
& {\bf b} \in L_\infty(I;L_\infty(\divv; \Omega)),\,\, e \in L_\infty(I \times \Omega),\,\, \essinf(e-{\textstyle \frac12} \divv_{\bf x} {\bf b}) \geq 0,
\end{split}
\ee
and $|\Gamma|>0$ when the latter $\essinf$ is zero, so that \eqref{boundedness} and \eqref{coercivity} are valid.
In this setting, the operators $A_a$, $A_s=A_s(\eps)$, and so $A=A(\eps)=A_s(\eps)+A_a$, are given by
\begin{align*}
(A_a w)(v)&=\int_I \int_\Omega({\bf b}\cdot\nabla_{\bf x} w+{\textstyle \frac12} w \divv_{\bf x} {\bf b}) v \dif{\bf x}\dif t,\\
(A_s(\eps) w)(v)&=\int_I \int_\Omega \eps \nabla_{\bf x} w \cdot \nabla_{\bf x} v+(e-{\textstyle \frac12}\divv_{\bf x} {\bf b})w v\dif{\bf x}\dif t.
\end{align*}
Thinking of ${\bf b}$ and $e$ fixed, and variable $\eps>0$, one infers that $\alpha=\alpha(\eps) \rightarrow \infty$ when $\eps \downarrow 0$ (cf. Remark~\ref{relative}).

In the next subsection we will construct $(X^\delta)_{\delta \in \Delta} \subset X$ and $(Y^\delta)_{\delta \in \Delta} \subset Y$
that (essentially) satisfy $\inf_{\eps>0} \gamma_\Delta^C(\eps)>0$
as families of finite element spaces w.r.t. subdivisions of $I \times \Omega$ into time-slabs with prismatic elements in each slab w.r.t.~generally different partitions of $\Omega$. Notice that although $C=\partial_t+A_a$ is independent of $\eps$, $\gamma_\Delta^C(\eps)$ depends on $\eps$ because it is defined in terms of the $\eps$-dependent energy-norm $\|\cdot\|_Y=\sqrt{(A_s(\eps)\cdot)(\cdot)}$.

As a consequence of $\gamma_\Delta^C(\eps)$ being uniformly positive,
for $K_Y^\delta \eqsim ({E_Y^\delta}' A_s E_Y^\delta)^{-1}$ uniformly in $\eps$ and $\delta$, i.e., $\sup_{\eps>0} \frac{\max(R_\Delta,1)}{\min(r_\Delta,1)}<\infty$,
Theorem~\ref{thm13} gives $\eps$-robust quasi-optimality results for MR and BEN w.r.t.~the $\eps$- (and $\beta$-) dependent norm $\nrm\cdot\nrm_X$.

\subsection{Realization of $\inf_{\eps} \gamma_\Delta^C(\eps)>0$}
Given a conforming partition $\tria$ of a polytopal $\overline{\Omega}$ into (essentially disjoint) closed $d$-simplices, we define $\mathcal{S}^{-1,q}_{\tria}$ as the space of all (discontinuous) piecewise polynomials of degree $q$ w.r.t.~$\tria$, and, for $q \geq 1$, set
$$
\mathcal{S}^{0,q}_{\tria,0}:=\mathcal{S}^{-1,q}_{\tria} \cap H_{0,\Gamma}^1(\Omega),
$$
where we assume that $\Gamma$ is the union of faces of $T \in \tria$.

Let $(\tria^\delta)_{\delta \in \bar{\Delta}}$,  $(\tria_S^\delta)_{\delta \in \bar{\Delta}}$ \new{be families} of such partitions of $\overline{\Omega}$ that are uniformly shape regular (which for $d=1$ should be read as to satisfy a uniform K-mesh property), and
where $\tria_S^\delta$ is a refinement of $\tria^\delta$ of some \emph{fixed} maximal depth in the sense that $|T| \gtrsim |T'|$ for
$\tria^\delta_S \ni T \subset T' \in \tria^\delta$, so that $\dim \tria_S^\delta \lesssim \dim \tria^\delta$.
On the other hand, fixing a $q \geq 1$, we require that the refinement from $\tria^\delta$ to $\tria_S^\delta$ is sufficiently deep that it permits the construction of a
projector $P_q^\delta$ for which
\begin{align} \label{35relax}
&\ran P_q^\delta \subseteq \mathcal{S}^{0,q}_{\tria_S^\delta,0},\quad\ran (\identity - P_q^\delta) \perp_{L_2(\Omega)} \big(\mathcal{S}^{0,q}_{\tria^\delta,0}+\mathcal{S}^{-1,q-1}_{\tria^\delta}\big),\\
 \label{34relax}
& \|P_q^\delta w\|_{L_2(T)}  \lesssim \|w\|_{L_2(T)} \quad (T\in \tria^\delta,\,w \in L_2(\Omega)).
\end{align}

As shown in \cite[Lemma~5.1 and Rem.~5.2]{58.6}, regardless of the refinement rule (e.g.~red-refinement \new{or} newest vertex bisection) that is (recursively) applied to create $(\tria_S^\delta)_{\delta \in \bar{\Delta}}$ from $(\tria^\delta)_{\delta \in \bar{\Delta}}$, there is a refinement of some fixed depth that suffices to satisfy \eqref{34relax} as well as
\be \label{35}
\ran P_q^\delta \subseteq \{w \in \mathcal{S}^{0,q}_{\tria_S^\delta,0}\colon w|_{\cup_{T \in \tria} \partial T}=0\},\quad\ran (\identity - P_q^\delta) \perp_{L_2(\Omega)} \mathcal{S}^{-1,q}_{\tria^\delta,0}.
\ee
Condition~\eqref{35} is stronger than \eqref{35relax}, and will be relevant in Sect.~\ref{Sapost} on robust a posteriori error estimation.

For $d \in \{1,2,3\}$ and $q \in \{1,2,3\}$, and both newest vertex bisection and red-refinement it was verified that it is sufficient that the aformentioned depth creates in the space $\mathcal{S}^{0,q}_{\tria_S^\delta,0}$ an additional number of degrees of freedom interior to any $T \in \tria^\delta$ that is greater or equal to ${q+d\choose q}$.

\begin{remark}
  To satisfy condition \eqref{35relax}--\eqref{34relax} generally a smaller number of degrees of freedom interior to any $T \in \tria^\delta$ suffices.
  For $d=2=q$, in \cite[Appendix A]{58.6} it was shown that in order to satisfy \eqref{35relax}--\eqref{34relax} it is sufficient to create $\tria_s^\delta$ from $\tria^\delta$ by one red-refinement, which creates only three of such degrees of freedom, whereas to satisfy \eqref{34relax}--\eqref{35} six additional interior degrees of freedom are needed.
\end{remark}

We show robustness of MR and BEN in a time-slab partition setting.

\begin{theorem} \label{thm:robust1} Let $H$, $V$, and $a(\cdot;\cdot,\cdot)$ be as in \eqref{bil}, with \emph{constant} ${\bf b}$ and \emph{constant} $e \geq 0$, and let $(\tria^\delta)_{\delta \in \bar{\Delta}}$ and  $(\tria_S^\delta)_{\delta \in \bar{\Delta}}$ be as specified above.
Then if, for $\delta \in \Delta$, $X^\delta$ and $Y^\delta$ are such that for some
finite partition $I^\delta=([t_{i-1}^\delta,t_{i}^\delta])_i$ of $I$, and arbitrary $\delta_i \in \bar{\Delta}$,
\begin{align}
X^\delta & \subseteq \{w \in C(I;H^1_{0,\Gamma}(\Omega)) \colon w|_{(t_{i-1}^\delta,t_{i}^\delta)} \in {\mathcal P}_q(t_{i-1}^\delta,t_{i}^\delta) \otimes {\mathcal S}^{0,q}_{\tria^{\delta_i},0}\},\\ \nonumber
Y^\delta & \supseteq\{v \in L_2(I;H^1_{0,\Gamma}(\Omega)) \colon v|_{(t_{i-1}^\delta,t_{i}^\delta)} \in {\mathcal P}_q(t_{i-1}^\delta,t_{i}^\delta) \otimes {\mathcal S}^{0,q}_{\tria_S^{\delta_i},0}\},
\end{align}
then $\inf_{\eps>0} \gamma_\Delta^C(\eps)>0$. Consequently the bounds \eqref{bndMR} and \eqref{bndBEN} show quasi-optimality of MR and BEN w.r.t.~the ($\eps$- and $\beta$-dependent) norm $\nrm\cdot\nrm_X$, uniformly in $\eps>0$ and $\beta\geq 1$.
\end{theorem}

\begin{proof} As follows from Proposition~\ref{39} the statement $\inf_{\eps>0} \gamma_\Delta^C(\eps)>0$ is equivalent to existence of  $Q_C^\delta \in \cL(Y,Y)$ with
\be \label{conds}
\sup_{\eps>0,\,\delta \in \Delta}\|Q_C^\delta\|_{\cL(Y,Y)} <\infty,\,
\ran Q_C^\delta \subset Y^\delta,\,
  \int_I \int_\Omega ((\partial_t +{\bf b} \cdot \nabla_{\bf x}) X^\delta) (\identity-Q_C^\delta) Y\dif {\bf x}\dif t=0,
\ee
  where we recall that, thanks to constant ${\bf b}$, $Y=L_2(I;H^1_{0,\Gamma}(\Omega))$ is equipped with norm
\begin{align*}
  \sqrt{(A_s(\eps) v)(v)}&= \sqrt{\int_I \eps\| \nabla_{\bf x} v\|_{L_2(\Omega)^d}^2+ e\|v\|_{L_2(\Omega)}^2\dif t} \\
  &\eqsim\sqrt{\eps} \|\nabla_{\bf x} v\|_{L_2(I \times \Omega)^d}+\sqrt{e} \| v\|_{L_2(I;L_2(\Omega))}.
\end{align*}

It holds that
\be \label{incl}
\hspace*{-0.3em}(\partial_t \!+\!{\bf b} \cdot \nabla_{\bf x}) X^\delta \subseteq
\big\{v \!\in\! L_2(I\! \times\! \Omega)\colon v|_{(t_{i-1}^\delta,t_i^\delta)} \!\in\! {\mathcal P}_q(t_{i-1}^\delta,t_i^\delta)\! \otimes\!
(\mathcal{S}^{0,q}_{\tria^{\delta_i},0}\!+\! \mathcal{S}^{-1,q-1}_{\tria^{\delta_i}})\big\}.\hspace*{-0.5em}
\ee
Let $(Q_{\bf x}^\delta)_{\delta \in \bar{\Delta}}$ denote a family of projectors such that
\begin{align} \label{33a}
&\sup_{\delta \in \Delta} \max\big( \|Q_{\bf x}^\delta\|_{\cL(L_2(\Omega),L_2(\Omega))},  \|Q_{\bf x}^\delta\|_{\cL(H^1_{0,\Gamma}(\Omega),H^1_{0,\Gamma}(\Omega))}\big)<\infty,\\ \label{33}
&\ran Q_{\bf x}^\delta \subset \mathcal{S}^{0,q}_{\tria_S^\delta,0},\quad \ran(\identity-Q_{\bf x}^\delta) \perp_{L_2(\Omega)}
\big(\mathcal{S}^{0,q}_{\tria^\delta,0}+\mathcal{S}^{-1,q-1}_{\tria^\delta}\big),
\end{align}
and let $Q^{\delta,i}$ be the $L_2(t_{i-1}^\delta,t_{i}^\delta)$-orthogonal projector onto ${\mathcal P}_q(t_{i-1}^\delta,t_{i}^\delta)$. Then, the operator
$Q_C^\delta$, defined by
$$
(Q_C^\delta v)|_{(t_{i-1}^\delta,t_{i}^\delta)\times \Omega}=(Q^{\delta,i} \otimes Q_{\bf x}^{\delta_i}) v|_{(t_{i-1}^\delta,t_{i}^\delta)\times \Omega},
$$
  satisfies \eqref{conds}. Indeed its uniform boundedness w.r.t.~the energy-norm on $Y$ follows by the boundedness of $Q_{\bf x}^\delta$ w.r.t.~both the $L_2(\Omega)$- and $H^1(\Omega)$-norms.
By writing $\identity-Q^{\delta,i} \otimes Q_{\bf x}^{\delta_i}=(\identity-Q^{\delta,i})\otimes \identity+Q^{\delta,i}\otimes(\identity-Q_{\bf x}^{\delta_i})$, and using \eqref{incl}
one verifies the third condition in \eqref{conds}.

We seek $Q_{\bf x}^\delta$ of the form $Q_{\bf x}^\delta=\breve{Q}_{\bf x}^{\delta}+\hat{Q}_{\bf x}^{\delta}+\hat{Q}_{\bf x}^{\delta} \breve{Q}_{\bf x}^{\delta}$
where
\be \label{36}
\ran \breve{Q}_{\bf x}^{\delta}, \ran \hat{Q}_{\bf x}^{\delta} \subset \mathcal{S}^{0,q}_{\tria_S^\delta,0},\quad
 \ran (\identity-\hat{Q}_{\bf x}^{\delta}) \perp_{L_2(\Omega)}
(\mathcal{S}^{0,q}_{\tria^\delta,0}+\mathcal{S}^{-1,q-1}_{\tria^\delta}).
\ee
Then from $\identity- Q_{\bf x}^\delta=(\identity- \hat{Q}_{\bf x}^{\delta})(\identity- \breve{Q}_{\bf x}^{\delta})$,
we infer that \eqref{33} is satisfied.

We take $\hat{Q}_{\bf x}^{\delta}=P_q^\delta$ from \eqref{35relax}--\eqref{34relax}.
It satisfies the properties required in \eqref{36}.
With $\hbar_\delta$ being the piecewise constant function defined by $\hbar_\delta|_{T}=\diam T$ $(T \in \tria^\delta)$,
thanks to the uniform $K$-mesh property of $\tria \in (\tria^\delta)_{\delta \in \bar{\Delta}}$,
\eqref{34relax} implies that
$\|\hbar_\delta^{-1} P_q^\delta \hbar_\delta \|_{\cL(L_2(\Omega),L_2(\Omega))}  \lesssim 1$,
as well as $\|P_q^\delta\|_{\cL(L_2(\Omega),L_2(\Omega))}  \lesssim 1$.

We take $\breve{Q}_{\bf x}^{\delta}$ as a modified Scott-Zhang quasi-interpolator onto $\mathcal{S}^{0,q}_{\tria_S^\delta,0}$ (\cite[Appendix]{75.527}). The modification consists in setting the degrees of freedom on $\Gamma$ to zero.
When applied to a function from $H^1_{0,\Gamma}(\Omega)$ it equals the original Scott--Zhang interpolator (\cite{247.2}), but thanks to the modification it is uniformly bounded w.r.t.~$L_2(\Omega)$, and so $\|Q_{\bf x}^\delta\|_{\cL(L_2(\Omega),L_2(\Omega))}$ is uniformly bounded.

Writing $Q_{\bf x}^\delta=\breve{Q}_{\bf x}^{\delta}+P_q^\delta(\identity-\breve{Q}_{\bf x}^{\delta})$,
from $\hbar_\delta^{-1}(\identity-\breve{Q}_{\bf x}^{\delta}) \in \cL(H^1_{0,\Gamma}(\Omega),L_2(\Omega))$,
$\hbar_\delta^{-1} P_q^\delta \hbar_\delta \in \cL(L_2(\Omega),L_2(\Omega))$, and
$\breve{Q}_{\bf x}^{\delta} \in \cL(H^1_{0,\Gamma}(\Omega),H^1_{0,\Gamma}(\Omega))$ all being uniformly bounded,
and $\|\cdot\|_{H^1(\Omega)} \lesssim \|\hbar_\delta^{-1} \cdot\|_{L_2(\Omega)}$ on $\mathcal{S}^{0,q}_{\tria_S^\delta,0}$, we infer the uniform boundedness of $ \|Q_{\bf x}^\delta\|_{\cL(H^1_{0,\Gamma}(\Omega),H^1_{0,\Gamma}(\Omega))}$.
\end{proof}

Next under the condition that $\essinf (e -\frac12 \divv_{\bf x} {\bf b})>0$, we consider the case of \emph{variable} ${\bf b}$ and $e$.
The scaling argument that was applied directly below Theorem~\ref{thm0} shows that
it is no real restriction to assume that $\essinf (e -\frac12 \divv_{\bf x} {\bf b})>0$.
Although we will not be able to show $\inf_{\eps>0} \gamma_\Delta^C(\eps)>0$, this inf-sup condition will be valid modulo a perturbation which can be dealt with using Young's inequality similarly as in the proofs of Theorems~\ref{thm:mainMRM} and ~\ref{thm:mainBEN}. It will result in $\eps$- (and $\beta$-) robust quasi-optimality results for MR and BEN similar as for constant ${\bf b}$ and constant $e \geq 0$.

\begin{theorem} \label{thm:t-dependent} Consider the situation of Theorem~\ref{thm:robust1}, but now without the assumption of ${\bf b}$ and $e$ being constants. Assume ${\bf b} \in W_\infty^1(I \times \Omega)^d$, $\essinf (e -\frac12 \divv_{\bf x} {\bf b}) >0$,
and, only for the case that ${\bf b}$ is time-dependent,
\be \label{14}
|t_{i-1}^\delta-t_i^\delta| \lesssim \max_{T \in \tria^{\delta_i}} \diam(T).
\ee
Then for MR and BEN it holds
\begin{align*}
\nrm u-u^\delta\nrm_X &\lesssim
{\textstyle \frac{\max(R_\Delta,1)}{\min(r_\Delta,1)}}
 \inf_{w \in X^\delta} \nrm u-w\nrm_X,\\
\nrm u-\bar{u}^\delta\nrm_X &\lesssim
\inf_{w \in X^\delta} \nrm u-w\nrm_X+ \inf_{v \in Y^\delta} \|u-v\|_Y,
\end{align*}
 \emph{uniformly} in $\eps>0$ and $\beta \geq 1$.
 \end{theorem}

 \begin{proof} As in the proof of Theorem~\ref{thm13}, we follow the proofs of Theorems~\ref{thm:mainMRM} (MR) and \ref{thm:mainBEN} (BEN). We only need to adapt the derivation of a lower bound for the expression in the second line of \eqref{e8}.

 With $\xi=\essinf (e -\frac12 \divv_{\bf x} {\bf b})$, it holds that
 $$
 \sqrt{\xi}\|\cdot\|_{Y'} \leq \|\cdot\|_{L_2(I \times \Omega)} \leq {\textstyle \frac{1}{\sqrt{\xi}}} \|\cdot\|_Y.
 $$
 Let ${\bf b}_{\delta}$ be the piecewise constant vector field defined by taking the average of ${\bf b}$ over each prismatic element $(t_{i-1}^\delta,t_i^\delta) \times T$ for $T \in \tria^{\delta_i}$.
 We use $w \mapsto {\bf b}_{\delta} \cdot \nabla_{\bf x} w$ to approximate $A_a$.
We have $\|{\bf b}-{\bf b}_{\delta}\|_{L_\infty((t_{i-1}^\delta,t_i^\delta) \times T)^d} \lesssim \diam(T)\|{\bf b}\|_{W_\infty^1((t_{i-1}^\delta,t_i^\delta) \times T)^d}$ by \eqref{14}.
 An application of the inverse inequality on the family of spaces $(\cS_{\tria,0}^{0,q})_{\tria \in \bar{\Delta}}$ shows that
for some constant $L>0$, for $w \in X^\delta$ it holds that
 $$
 \|({\bf b}-{\bf b}_{\delta})\cdot \nabla_{\bf x} w+{\textstyle \frac12} w \divv_{\bf x} {\bf b}\|_{L_2(I\times\Omega)} \leq L \|w\|_{L_2(I \times \Omega)}.
 $$

 Because \eqref{incl} is also valid for \emph{piecewise} constant ${\bf b}$, and
 $$
   \sqrt{(A_s(\eps) v)(v)} \eqsim \sqrt{\eps} \|\nabla_{\bf x} v\|_{L_2(I \times \Omega)^d}+\sqrt{\xi} \| v\|_{L_2(I;L_2(\Omega))},
$$
 only dependent on $\|e -\frac12 \divv_{\bf x} {\bf b}\|_{L_\infty(I \times \Omega)}/\xi$,
  the proof of Theorem~\ref{thm:robust1} shows that for some constant $\gamma>0$, for $w \in X^\delta$ it holds that
 $$
 \|{E_Y^\delta}' (\partial_t+{\bf b}_{\delta} \cdot \nabla_{\bf x})E_Y^\delta w\|_{{Y^\delta}'} \geq \gamma
  \|(\partial_t+{\bf b}_{\delta} \cdot \nabla_{\bf x})E_Y^\delta w\|_{Y'}.
  $$

  By combining these estimates, we find that for $w \in X^\delta$ it holds that
  \begin{align*}
  \|{E_Y^\delta}' C E_Y^\delta w\|_{{Y^\delta}'} & \geq  \gamma \|(\partial_t+{\bf b}_{\delta} \cdot \nabla_{\bf x})E_Y^\delta w\|_{Y'} -{\textstyle \frac{L}{\sqrt{\xi}}} \|E_Y^\delta w\|_{L_2(I \times \Omega)}\\
  &  \geq
  \gamma \|C E_Y^\delta w\|_{Y'} -(\gamma+1) {\textstyle \frac{L}{\sqrt{\xi}}} \|E_Y^\delta w\|_{L_2(I \times \Omega)}\\
   &  \geq
  \gamma \|C E_Y^\delta w\|_{Y'} -(\gamma+1) {\textstyle \frac{L}{\xi}} \|E_Y^\delta w\|_{Y},
  \end{align*}
  and so
    \begin{align*}
     &\|{E_Y^\delta}' C E_Y^\delta w\|_{{Y^\delta}'}^2+\|E_X^\delta w\|_Y^2+\|\gamma_T E_X^\delta w\|^2+(\beta-1)\|\gamma_0 E_X^\delta w\|^2\\
 &  \geq \big( \gamma \|C E_Y^\delta w\|_{Y'} \!-\!(\gamma+1) {\textstyle \frac{L}{\xi}} \|E_Y^\delta w\|_{Y}\big)^2+\|E_X^\delta w\|_Y^2+\|\gamma_T E_X^\delta w\|^2 +(\beta-1)\|\gamma_0 E_X^\delta w\|^2\\
  & \geq (1-\eta^2) \gamma^2 \|C E_Y^\delta w\|_{Y'}^2+\big\{(1-\eta^{-2})(\gamma+1)^2 {\textstyle \frac{L^2}{\xi^2}}+1\big\}
  \|E_X^\delta w\|_Y^2+\\
  & \hspace*{7.8cm}\|\gamma_T E_X^\delta w\|^2 +(\beta-1)\|\gamma_0 E_X^\delta w\|^2.
  \end{align*}
 Minimizing over $\eta$ shows that, with  $\alpha^2:=(\gamma+1)^2 {\textstyle \frac{L^2}{\xi^2}}$, the last expression is greater than or equal to
 $$
 {\textstyle \frac12}\Big(\gamma^2+\alpha^2+1-\sqrt{(\gamma^2+\alpha^2+1)^2-4 \gamma^2}\Big)
 \nrm E_X^\delta w\nrm_X^2,
 $$
which completes the proof.
  \end{proof}

 \new{The undesirable condition \eqref{14} for time-dependent ${\bf b}$ might be pessimistic in practice, which however we have not tested so far.}

  \subsection{Robust a posteriori error estimation} \label{Sapost}
  A robust error estimator will be realized in the following limited setting.

 Consider the spaces and bilinear form $a$ as in \eqref{bil}, where ${\bf b}$ is constant, $e=0$, and the polytope $\Omega \subset \R^d$ is convex.
For families of quasi-uniform partitions $(I^\delta)_{\delta \in \Delta}$ of $\overline{I}$, and
 $(\tria^\delta)_{\delta \in \Delta}$ and  $(\tria_S^\delta)_{\delta \in \Delta}$ of $\overline{\Omega}$ as before, where $\tria_S^\delta$ is a sufficiently deep refinement of $\tria^\delta$ that permits the construction of a projector $P_1^\delta$ that satisfies \eqref{34relax}--\eqref{35},
 and for some $h_\delta>0$, $\diam T \eqsim h_\delta \eqsim \diam J$ ($T \in \tria^\delta,\,J \in I^\delta$),
 let $X^\delta:=S^{0,1}_{I^\delta} \otimes S^{0,1}_{\tria^\delta,0}$ and $Y^\delta:=S^{-1,1}_{I^\delta} \otimes S^{0,1}_{\tria_S^\delta,0}$.
 For completeness, $S^{-1,1}_{I^\delta}$ denotes the space of piecewise linears w.r.t.~$I^\delta$, and $S^{0,1}_{I^\delta}$ the space of continuous piecewise linears w.r.t.~$I^\delta$.

 In this setting, in \cite[Thm.~5.6]{58.6} projectors $Q_B^\delta \in \cL(Y,Y)$ have been constructed with $\ran Q_B^\delta \subset Y^\delta$ and $(\identity-{ Q_B^\delta }') B X^\delta=0$.
 Moreover, these $Q_B^\delta$ are uniformly bounded in $Y=L_2(I;H^1_{0,\Gamma}(\Omega))$ equipped with the standard Bochner norm, with $H^1_{0,\Gamma}(\Omega)$ being equipped with $\|\nabla \cdot\|_{L_2(\Omega)^d}$.
 Since for the current bilinear form $a$, the energy-norm $\|\cdot\|_Y$ is equal to  $\sqrt{\eps}\|\cdot\|_{L_2(I;H^1_{0,\Gamma}(\Omega))}$, it holds that $\sup_{\delta \in \Delta,\,\eps>0} \|Q_B^\delta\|_{\cL(Y,Y)}<\infty$, and so
 $$
 \inf_{\eps>0} \gamma_\Delta^B(\eps)>0.
 $$

Let $((\hat{K}_Y^\delta)^{-1}v)(v) \eqsim \int_I \int_\Omega |\nabla_{\bf x} v|^2 \dif {\bf x} \dif t$ ($\delta \in \Delta,\, v \in Y^\delta$), then $(\eps^{-1} \hat{K}_Y^\delta)^{-1} \eqsim {E_Y^\delta}' A_s E_Y^\delta$, i.e., using preconditioner $K_Y^\delta:=\eps^{-1} \hat{K}_Y^\delta$ it holds that $\sup_{\eps>0}\frac{\max(R_\Delta,1)}{\min(r_\Delta,1)}<\infty$.

 What remains is to show that \emph{data-oscillation} is asymptotically of higher or equal order as the approximation error
 in $\nrm\cdot\nrm_X=\sqrt{\|B\cdot\|_{Y'}^2+\beta\|\gamma_0\cdot\|^2}$. Noting that $\|\cdot\|_{Y'} = \frac{1}{\sqrt{\eps}} \|\cdot\|_{L_2(I;H^1_{0,\Gamma}(\Omega)')}$, it is natural to select
 $$
 \beta =\eps^{-1}.
 $$
 Then $\sqrt{\eps} \nrm  \cdot \nrm_X$ equals
 $$
 \sqrt{\|(\partial_t+{\bf b}\cdot\nabla_{\bf x})\cdot\|_{L_2(I;H^1_{0,\Gamma}(\Omega)')}^2+\eps^2\|\cdot\|_{L_2(I;H^1_{0,\Gamma}(\Omega))}^2+\eps \|\gamma_T \cdot\|^2+(1-\eps)\|\gamma_0\cdot\|^2},
 $$
 and so even for a general smooth $u$, $\sqrt{\eps}$ times the approximation error cannot be expected to be smaller than $\eqsim h_\delta^2$. Since for
 $g \in L_2(I;H^1(\Omega)) \cap H^2(I;H^{-1}(\Omega))$ it holds that
$\sqrt{\eps} \|(\identity-{Q_B^\delta}')g\|_{Y'}= \|(\identity-{Q_B^\delta}')g\|_{L_2(I;H^1_{0,\Gamma}(\Omega)')} \lesssim h_\delta^2$ (\cite[Thm.~5.6]{58.6}),
 we conclude that $\cE^\delta(w;g,u_0,\beta)$ from~\eqref{eqn:apost} is an efficient and, modulo above satisfactory data-oscillation term, reliable a posteriori estimator of the error in $w$ in $\nrm  \cdot \nrm_X$-norm.

\section{Numerical test} \label{Snumer}
We tested the minimal residual (MR)  method applied to the parabolic initial value
problem with the singularly perturbed `spatial component' as given in \eqref{bil}.
We considered the simplest case where $I=\Omega=(0,1)$, ${\bf b}=1$, and $e$ is
either $0$ or $1$, and $X^\delta=S_{I^\delta}^{0,1} \otimes S_{\tria^\delta,0}^{0,1}$,
where $I^\delta=\tria^\delta$ is a uniform partition of the unit interval with mesh
size $h_\delta$.  Taking always $(K_Y^\delta)^{-1}={E_Y^\delta}' A_s E_Y^\delta$,
we took either
{\renewcommand{\theenumi}{\roman{enumi}}
\begin{enumerate}
\item \label{i}
  $Y^\delta=S_{I^\delta}^{-1,1} \otimes S_{\tria^\delta,0}^{0,1} (\supseteq X^\delta \cup \partial_t X^\delta)$
  which for any \emph{fixed} $\eps>0$ gives $\gamma_\Delta^{\partial_t}>0$ (Sect.~\ref{full}),
  so that the MR approximations are quasi-optimal approximations from the trial space
  w.r.t.~$\|\cdot\|_X$ (Thm.~\ref{thm:mainMRM}), or
\item \label{ii}
  $Y^\delta:=S_{I^\delta}^{-1,1} \otimes S_{\tria_s^\delta,0}^{0,1}$ where $\tria_s^\delta$
  is a uniform partition with mesh-size $h_{\delta}/3$ which even gives $\inf_{\eps>0} \gamma_\Delta^{C}(\eps)>0$
  (Thm.~\ref{thm:robust1}), so that the MR approximations are quasi-optimal
  approximations from the trial space w.r.t.~the energy-norm $\nrm\cdot\nrm_X$
  also \emph{uniformly in $\eps>0$} (Thm.~\ref{thm13}).
\end{enumerate}
Remark~\ref{rem10} shows that in these cases the BEN and MR methods give the same solution.

As discussed in Sect.~\ref{Sapost}, for the case that $e=0$ it is natural to take
the weight $\beta=\eps^{-1}$.  Unlike with $e=0$, for $e=1$ and $0 \neq v \in Y$
the energy-norm $\sqrt{(A_s v)(v)}$ does not tend to zero for $\eps \downarrow 0$
but converges to $\|v\|_{L_2(I \times \Omega)}$. In view of this there is no reason
to let $\beta$ tend to infinity for $\eps \downarrow 0$, and we took $\beta=1$.

For $Y^\delta$ as in \eqref{ii}, in Sect.~\ref{Sapost} it was shown that for
$(e,\beta)=(0,\eps^{-1})$ it holds that $\inf_{\eps>0} \gamma_\Delta^{B}(\eps)>0$,
and more specifically that the a posteriori error estimator $\cE^\delta(w;g,u_0,\beta)$ from~\eqref{eqn:apost}
is an efficient and, modulo a data-oscillation term that is at least of equal order,
reliable estimator of the error $\nrm u-w \nrm_X$.  Therefore to assess our numerical
results, we used $Y^\delta$ as in Option~\eqref{ii} for error estimation,
even when solving with $Y^\delta$ as in \eqref{i}.

For $(e,\beta)=(1,1)$,  we numerically \emph{observed} that for our model problems
the a posteriori error estimator $\cE^\delta(w;g,u_0,\beta)$ computed with $Y^\delta$ as in \eqref{ii}
is efficient and reliable as, knowing that the estimator \emph{equals} $\nrm u - w\nrm_X$
for $Y^\delta = Y$, we saw that further overrefinement of the
test space $Y^\delta$ never increased the estimated error by more than a percent.
So again, regardless of whether we took $Y^\delta$
as in Option~\eqref{i} or \eqref{ii}, we used $Y^\delta$ as in \eqref{ii} to compute $\cE^\delta(w; g, u_0, \beta)$.

In experiments below, we choose $\eps = 1, 10^{-1}, 10^{-3}, 10^{-6}$; to compare different values of $\eps$,
we show the estimated error divided by an accurate approximation for $\sqrt{\|g\|_{Y'}^2+\beta \|u_0\|^2}$, which
is equal to the $\nrm\cdot\nrm_X$-norm of the exact solution.

\subsection{Smooth problem}
We take (homogeneous) Dirichlet boundary conditions at left- and right boundary, i.e.~$\Gamma=\partial\Omega$, select $(e, \beta) = (0, \eps^{-1})$, and prescribe the exact solution
$u(t,x) := (t^2 + 1) \sin(\pi x)$ with derived data $u_0$ and $g$.
For this problem, the best possible error in $\nrm\cdot\nrm_X$-norm, divided by $\nrm u\nrm_X$,
decays proportionally to $(\dim X^\delta)^{-1/2}$.

Figure~\ref{fig:smooth} shows this relative estimated error as a function of $\dim X^\delta$.
In accordance with Theorem~\ref{thm:mainMRM}, for this parabolic problem with non-symmetric spatial part, both Option~\eqref{i} and Option~\eqref{ii} give solutions that converge at the expected rate. For Option~\eqref{i}, however, this convergence is not uniform in $\eps$, but in accordance with Theorem~\ref{thm13}, for Option~\eqref{ii} it is.

\begin{figure}[ht]
  \includegraphics[width=\linewidth]{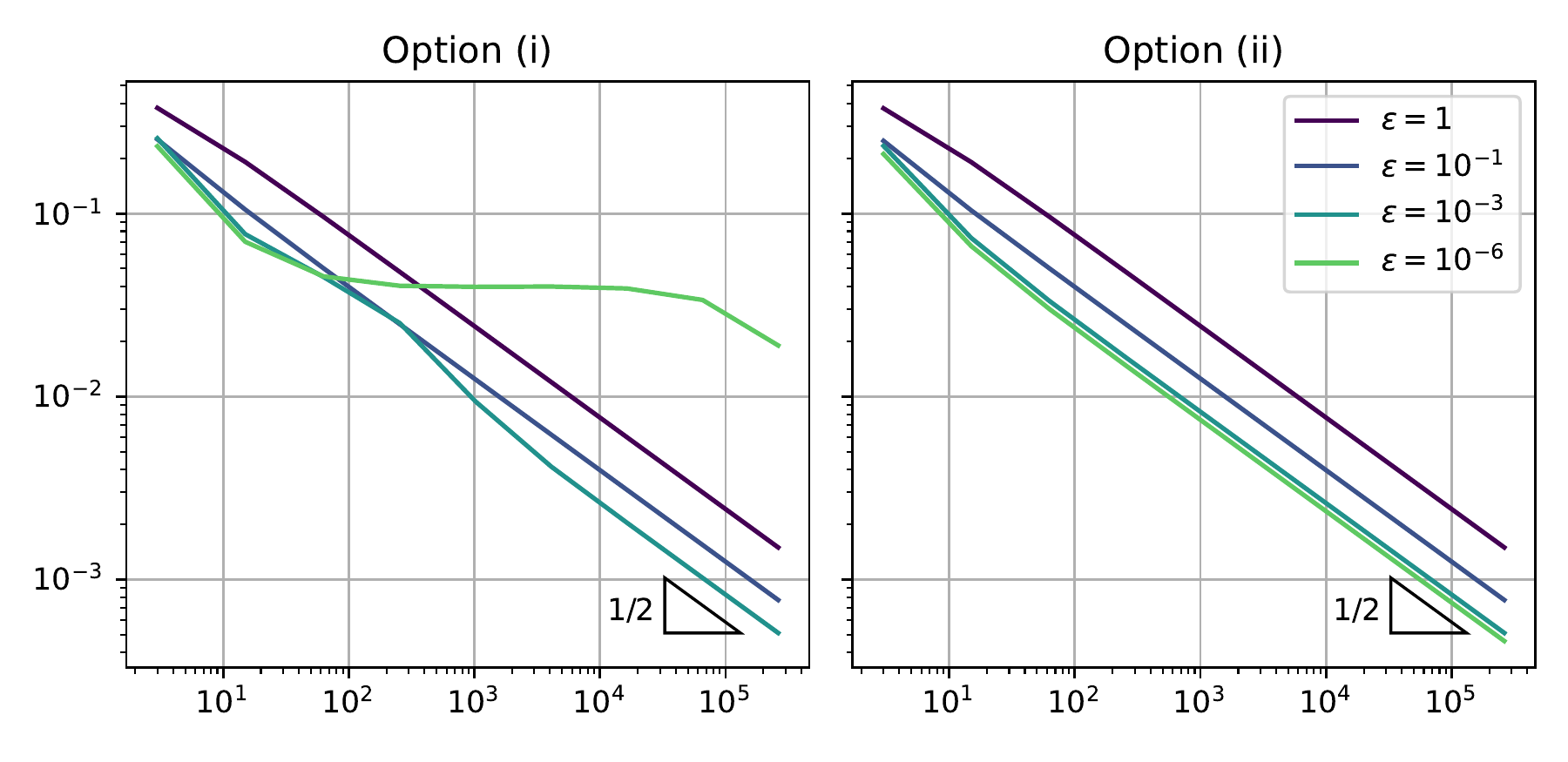}
  \caption{Relative estimated error progression for the \emph{smooth problem} as function of $\dim X^\delta$ for different diffusion rates $\eps$. Left: test space $Y^\delta$ as in Option~(\ref{i}); right: $Y^\delta$ as in (\ref{ii}).}
  \label{fig:smooth}
\end{figure}

\subsection{Internal layer problem}
We choose $u_0 := 0$ and $g(t,x) := \mathbbm{1}_{\{x > t\}}$, select $(e, \beta) = (0, \eps^{-1})$, and prescribe a homogeneous Dirichlet boundary condition only at the left boundary $x=0$, i.e.~$\Gamma := \{0\}$, and so have a Neumann boundary condition at the `outflow' boundary $x=1$. Due to the jump in the forcing data, in the limit $\eps \downarrow 0$, the solution $t \cdot \mathbbm{1}_{\{x>t\}}$ is discontinuous along the diagonal $x=t$.

The left of Figure~\ref{fig:internal} shows the relative estimated error progression of Option~(\ref{ii}) as a function of $\dim X^\delta$; as Option~(\ref{i}) again suffers from degradation for small $\eps$ (with results very similar to the left of Figure~\ref{fig:smooth}), we omit a graph of its error progression.
Its right shows the discrete solution at $h_\delta = \tfrac{1}{512}$ and $\eps=10^{-6}$. The solution resembles the \emph{pure transport} solution quite well, with the exception of a small artefact near $x=t=0$.

\begin{figure}[ht]
  \includegraphics[width=\linewidth]{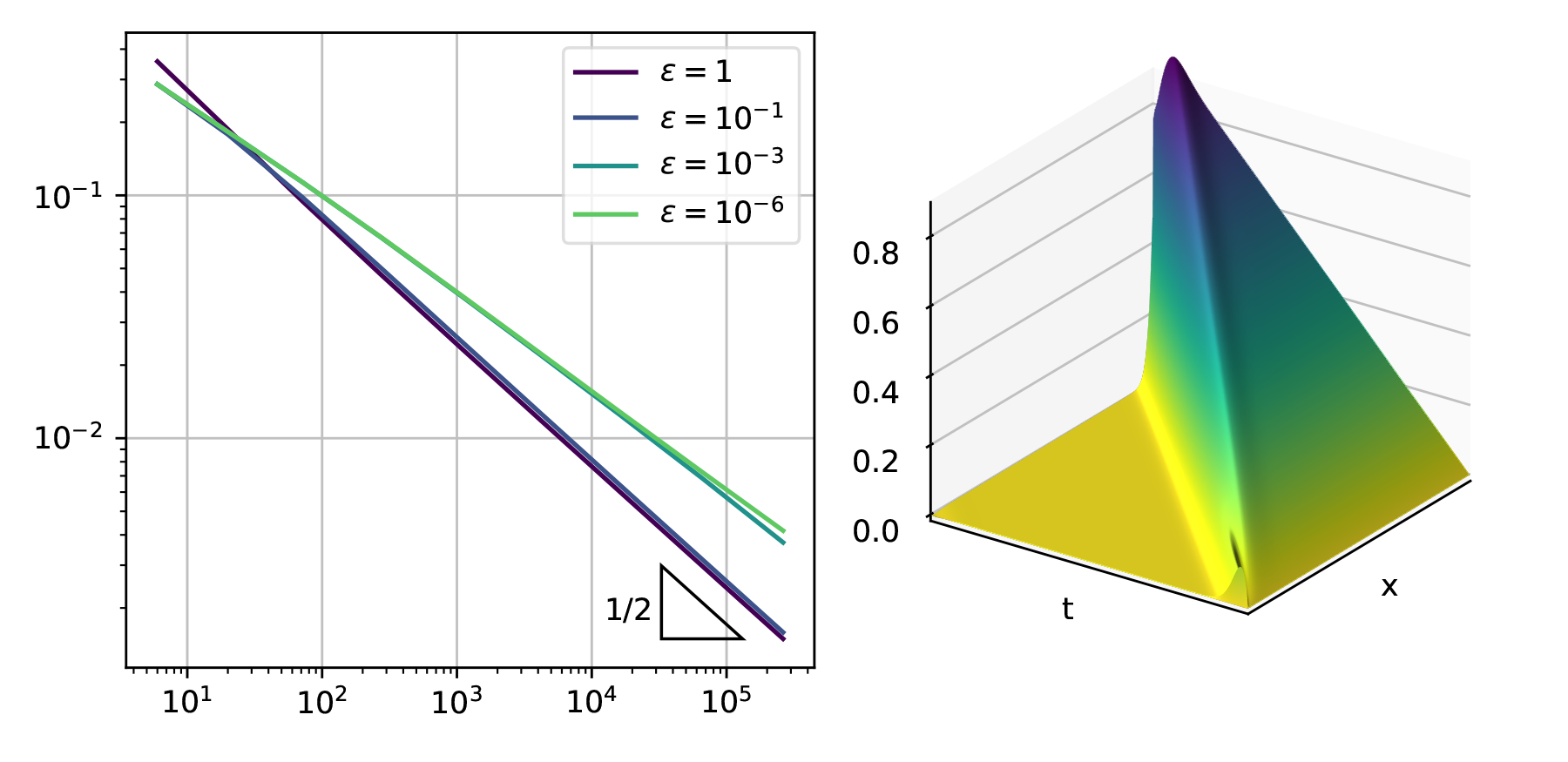}
  \caption{Solving the \emph{internal layer problem} with Option~(\ref{ii}). Left: relative estimated error progression as function of $\dim X^\delta$ for different diffusion rates $\eps$. Right: solution at $h_\delta = \tfrac{1}{512}$ and $\eps=10^{-6}$.}
  \label{fig:internal}
\end{figure}

\subsection{Boundary layer problem}
We choose $u_0(x) := \sin(\pi x)$ and $g = 0$, select $(e, \beta) = (1, 1)$, and set homogenous Dirichlet boundary conditions on $\partial \Omega$, i.e.~$\Gamma = \{0,1\}$.
Due to the condition on the outflow boundary, the problem is ill-posed in the limit $\eps = 0$, hence for $\eps$ small, the solution has a boundary layer at $x=1$.

Figure~\ref{fig:bdrlayer1} shows that the method fails to make progress
until the boundary layer is resolved at $h_\delta \lesssim \eps$.
Figure~\ref{fig:bdrlayer2} shows two discrete solutions at $h_\delta = \tfrac{1}{512}$ computed for Option~(\ref{ii}). We see that for $\eps=10^{-3}$, the boundary layer is resolved and the solution resembles the \emph{pure transport} solution quite well, with the exception of a small artefact near $x=t=1$.
For $\eps = 10^{-6}$ though, the boundary layer cannot be resolved with the current (uniform) mesh, and the solution is completely wrong.
For $\eps \downarrow 0$, the energy-norm of the error in an approximation $w$ approaches $\sqrt{\|(\partial_t+{\bf b}\cdot \nabla_{\bf x})w\|_{L_2(I \times \Omega)}^2+\|u_0-\gamma_0 w\|_{L_2(\Omega)}^2}$.
As a result, for streamlines that hit the outflow boundary, the method `chooses' to smear the unavoidably large error as a consequence of the layer along the whole streamline resulting in a globally bad approximation.
This is a well-known phenomenon when using a least squares method to approximate a solution that has a sharp layer or a shock.

\begin{figure}[ht]
  \includegraphics[width=\linewidth]{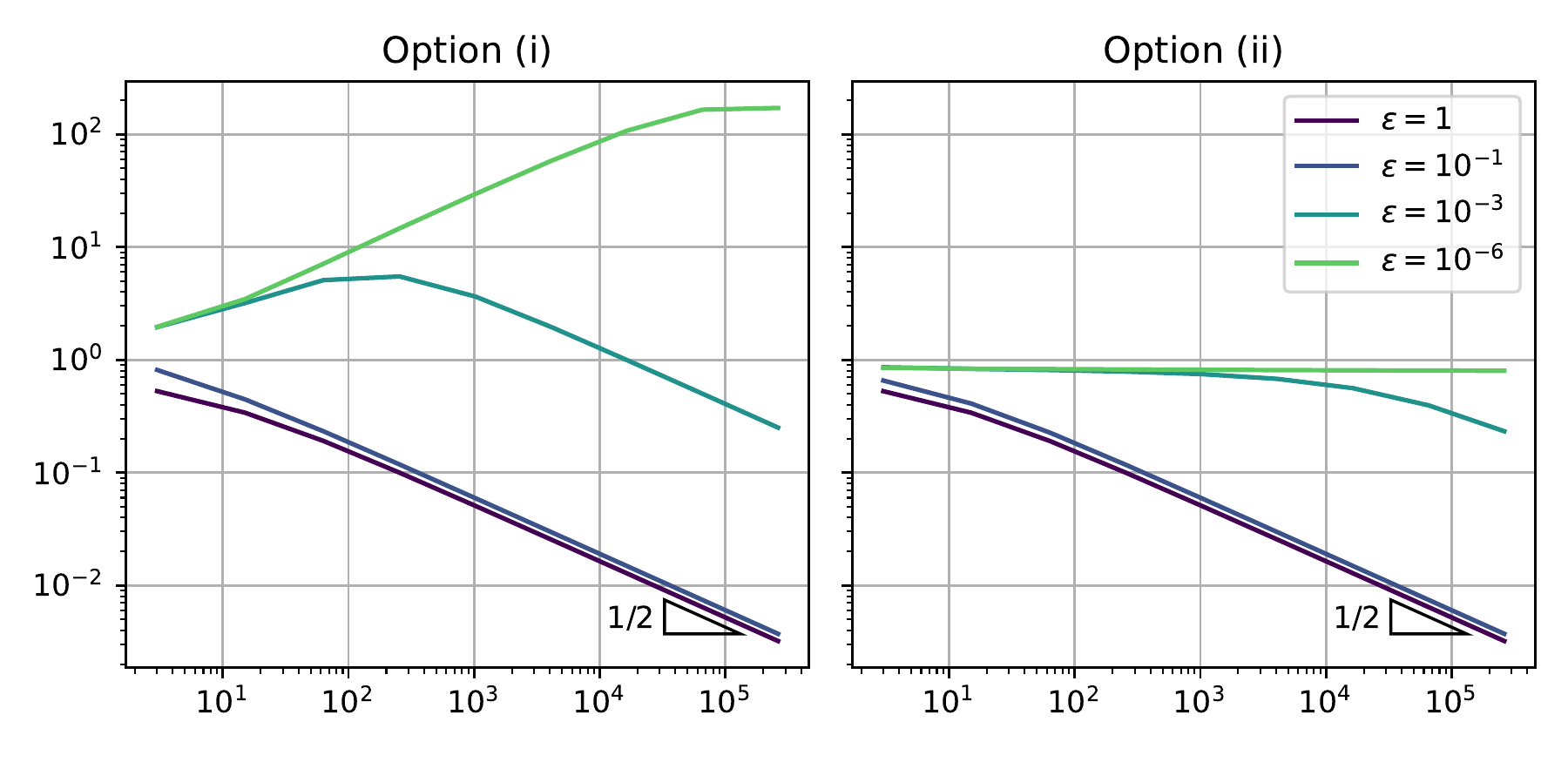}
  \caption{Relative estimated error progression for the \emph{boundary layer problem} as function of $\dim X^\delta$ for different diffusion rates $\eps$. Left: test space $Y^\delta$ as in Option~(\ref{i}); right: $Y^\delta$ as in (\ref{ii}).}
  \label{fig:bdrlayer1}
\end{figure}
\begin{figure}[ht]
  \includegraphics[width=\linewidth]{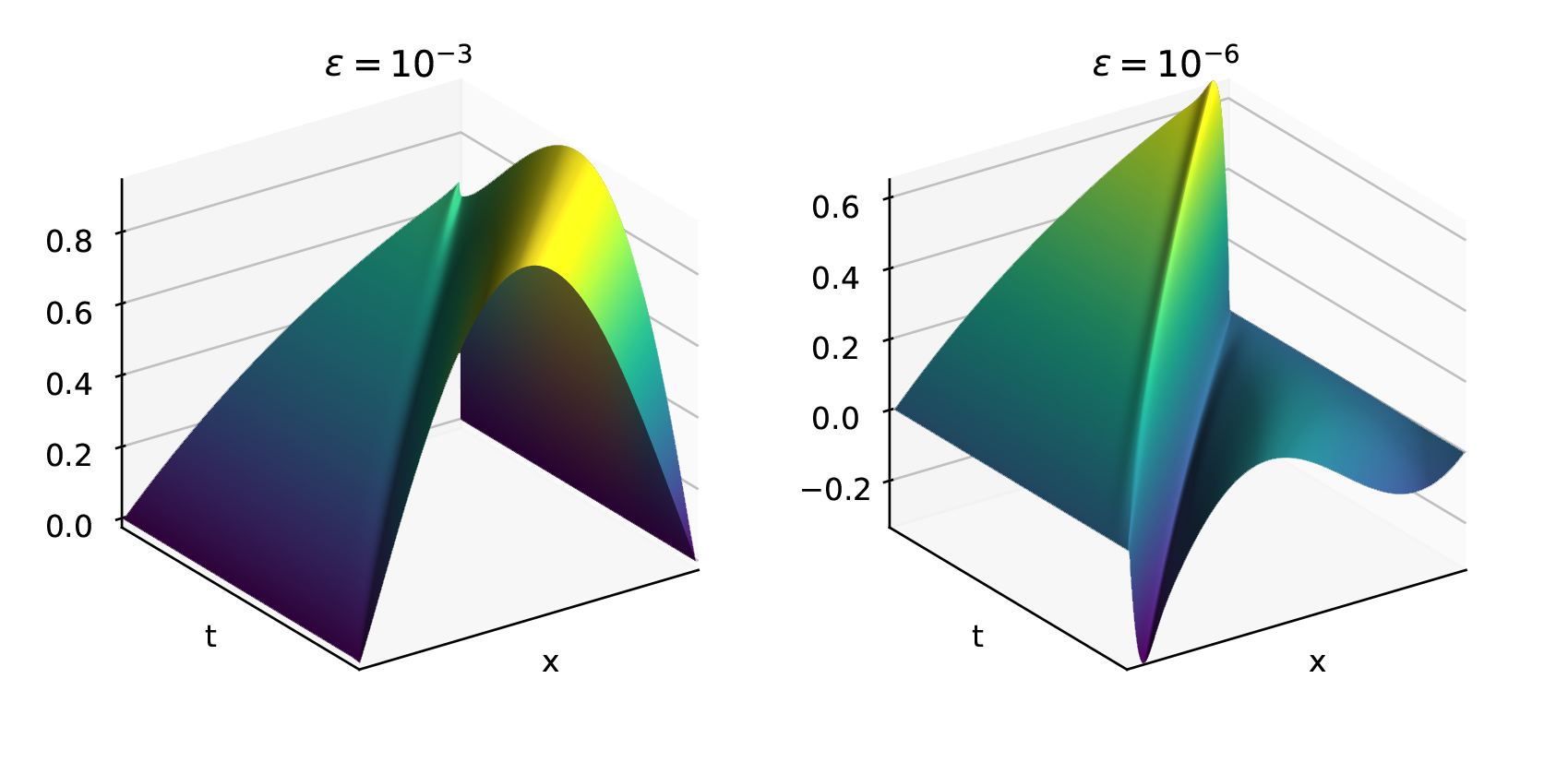}
  \caption{Solutions of the \emph{boundary layer problem} with Option~(\ref{ii}) at $h_\delta = \tfrac{1}{512}$. Left: diffusion $\eps = 10^{-3}$; right: $\eps = 10^{-6}$.}
  \label{fig:bdrlayer2}
\end{figure}

\subsection{Imposing outflow boundary conditions weakly}
One common work-around to the problem caused by the boundary layer is to refine the mesh strongly towards this layer.
An alternative is to impose at the outflow boundary the Dirichlet boundary condition only weakly, see e.g.~the references \cite{45.44,35.8565,35.9375,35.943} where this approach has been applied with least squares methods for stationary convection dominated convection-diffusion methods. \new{This approach is also known from other contexts, as in \cite{19.894,35.858,245.69}.}
Without having a rigorous analysis we tried \new{this weak imposition of the Dirichlet boundary condition }by computing, with $Y^\delta$ as in Option~(\ref{ii}),
\[
  u^\delta := \argmin_{w \in \hat{X}^\delta} \|{E_Y^\delta}'(BE_X^\delta w - g)\|_{{Y^\delta}'}^2 + \beta\|\gamma_0 E_X^\delta w - u_0\|^2 + \eps \|w(\cdot, 1)\|_{L_2(I)}^2.
\]
Here, $\hat{X}^\delta$ denotes the space $X^\delta$ after removing the Dirichlet boundary condition at $x=1$.
Figure~\ref{fig:bdrlayer3} shows the resulting error progression, which is robust in $\eps$, as well as the minimal residual solution at $h_\delta = \tfrac{1}{512}$ and $\eps=10^{-6}$; it resembles the \emph{pure transport} solution quite well, and does not suffer from the artifact present at the right of Figure~\ref{fig:bdrlayer2}.
\begin{figure}[ht]
  \includegraphics[width=\linewidth]{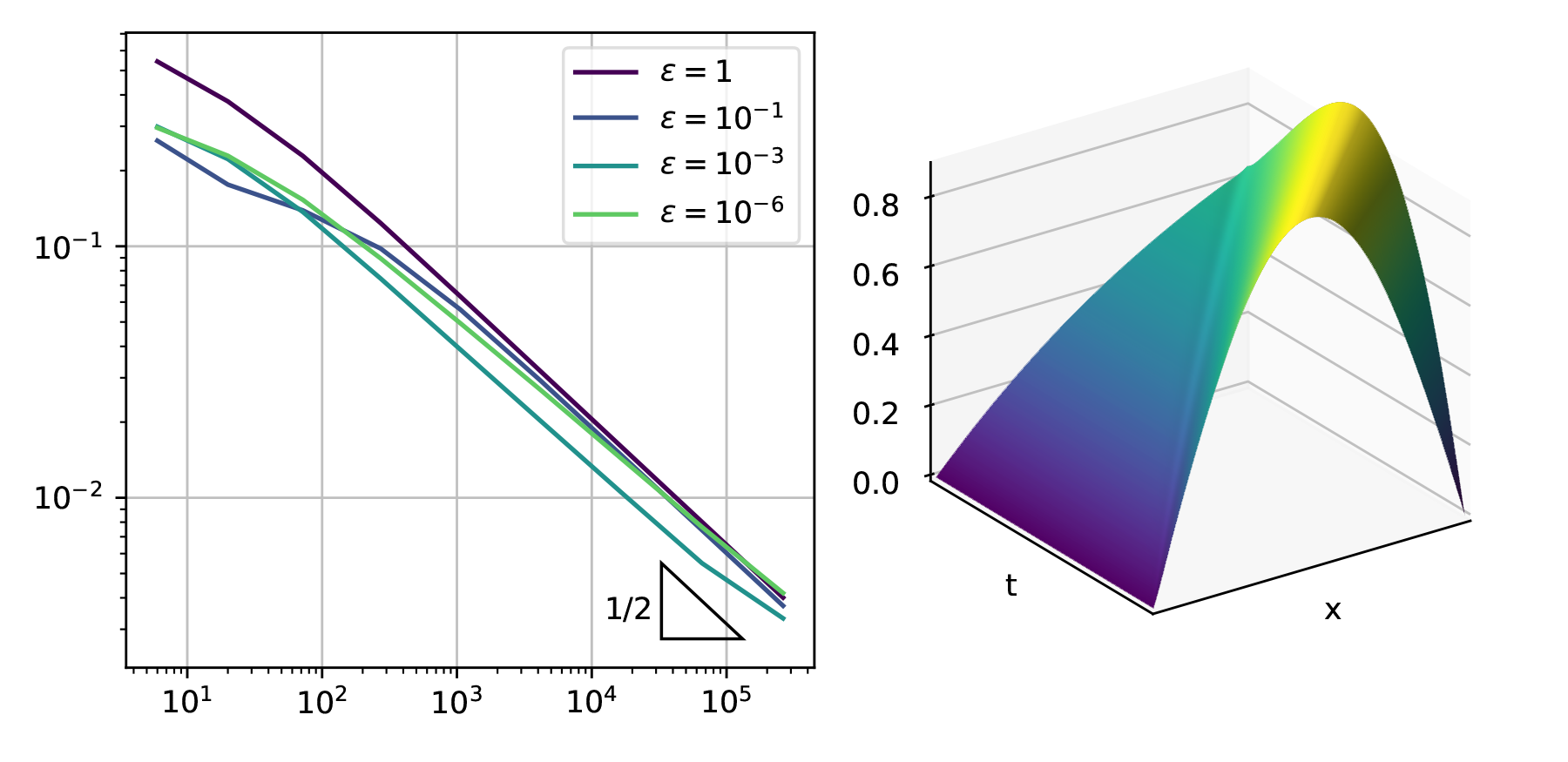}
  \caption{Solving the \emph{boundary layer problem} with Option~(\ref{ii}), and imposing the outflow boundary condition \emph{weakly}. Left: relative estimated error progression as function of $\dim \hat X^\delta$ for different diffusion rates $\eps$. Right: solution at $h_\delta = \tfrac{1}{512}$ and $\eps = 10^{-6}$.}
  \label{fig:bdrlayer3}
\end{figure}}


\end{document}